\documentclass[preprint,12pt]{temp}

\usepackage{amssymb}

\usepackage{amsthm}

\usepackage{amsmath}

\usepackage{tikz}

\usetikzlibrary{decorations.pathreplacing,calligraphy}

\usepackage{algorithm}
\usepackage[noend]{algorithmic}

\usepackage{longtable}

\usepackage{booktabs}

\usepackage{setspace}

\newtheorem{theorem}{Theorem}
\newtheorem{lemma}[theorem]{Lemma}

\newcommand{\med}{M}

\usepackage{lineno}

\journal{}

\begin{document}

\begin{frontmatter}



\title{Algorithms for Shipping Container Delivery Scheduling}

\author{Anna Collins\fnref{kcl}} 
\author{Dimitrios Letsios\fnref{kcl}} 
\author{Gueorgui Mihaylov\fnref{hal}} 



\address[kcl]{Department of Informatics, King's College London, United Kingdom}
\address[hal]{Haleon plc, United Kingdom}


\begin{abstract}
Motivated by distribution problems arising in the supply chain of Haleon, we investigate a discrete optimization problem that we call the \emph{container delivery scheduling problem}.
The problem models a supplier dispatching ordered products with shipping containers from manufacturing sites to distribution centers, where orders are collected by the buyers at agreed due times. 
The supplier may expedite or delay item deliveries to reduce transshipment costs at the price of increasing inventory costs, as measured by the number of containers and distribution center storage/backlog costs, respectively.
The goal is to compute a delivery schedule attaining good trade-offs between the two. 
This container delivery scheduling problem is a temporal variant of classic bin packing problems, where the item sizes are not fixed, but depend on the item due times and delivery times.
An approach for solving the problem should specify a batching policy for container consolidation 
and a scheduling policy for deciding when each container should be delivered.
Based on the available item due times, we develop algorithms with sequential and nested batching policies as well as on-time and delay-tolerant scheduling policies. 
We elaborate on the problem's hardness and substantiate the proposed algorithms with positive and negative approximation bounds, including the derivation of an algorithm achieving an asymptotically tight 2-approximation ratio. 
\end{abstract}




\end{frontmatter}


\section{Introduction}


We consider a delivery scheduling problem relevant to optimizing transportation and inventory costs incurred by deliveries from a supplier to a set of customers in a supply chain \cite{Fransoo2013,Jula2011,Levi2008}.
Customers (e.g.\ retailers) issue orders to the supplier (e.g.\ manufacturing site). 
Each order requests different amounts from a portfolio of products and is collected by the customer at a distribution center and at an agreed due time.
The products of one order must be simultaneously transported as a whole to the distribution center.
Figure~\ref{Fig:Setting} depicts this setting.
Because each container delivery is associated with a fixed cost, the supplier uses  consolidation, i.e.\ may combine multiple orders in one container, to reduce the number of used containers \cite{Tsertou2016}.  
However, this may result in deliveries to the distribution center not well synchronized with the corresponding order due times, thus inventory costs.
Specifically, if an order arrives before the agreed due time, then a storage cost is incurred.
Likewise, if an order arrives after the agreed due time, then a backlog cost in incurred.
Our objective is to decide when shipments should take place and the orders that each shipment should include so that the number of shipments is minimized and the total inventory cost is bounded.

Tolerating a bound on the total inventory cost allows reducing transportation costs, but may incur significant storage or backlog costs for a subset of items, thus can be problematic.
On one hand, lengthy storage periods, especially in the case of perishable goods, deteriorate product quality.
On the other hand, substantial delivery delays, especially in the case of new buyers, may result in customer attrition or additional mitigation costs.
As a preventive measure, we impose an upper bound on the total inventory cost incurred by any subset of orders that are shipped together.
In our application context, this bound can be used by the manufacturer to compute trade-offs between transportation and inventory costs for production planning purposes and for negotiating purchase order terms and conditions.

\begin{figure}[t!]
\begin{center}
\includegraphics[scale=0.55]{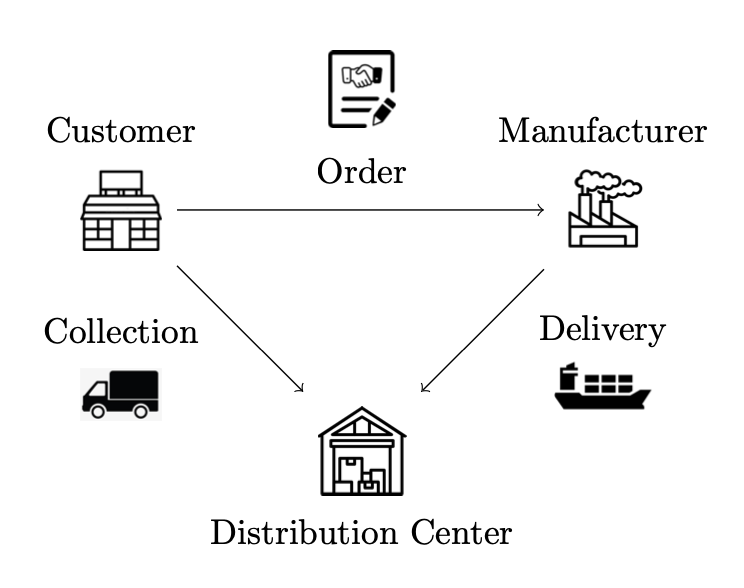}
\end{center}
\vspace*{-0.5cm}
\caption{Container delivery scheduling problem setting.}
\label{Fig:Setting}
\end{figure}


More formally, the \emph{container delivery scheduling problem} can be described as follows.
There is a set $\mathcal{I}$ of items to be transferred from a manufacturing site to a distribution center.
An item $i\in \mathcal{I}$ contains all products purchased by one customer order and is associated with a due time $d_i$. 
The goal is to assign each item $i$ to a container (or simply \emph{bin}).
If a bin delivery occurs at time $\tau$, then the subset $\mathcal{S}\subseteq\mathcal{I}$ of items in the bin are available in the distribution center at $\tau$ for collection by the customers.
In such a case, an inventory cost $\sum_{i\in\mathcal{S}}|d_i-\tau|$ is incurred, which corresponds to a storage cost $(d_i-\tau)$ for each item such that $d_i>\tau$, and a backlog cost $\tau-d_i$ for any item $i$ with $d_i<\tau$. 
The objective is to assign the items into a minimal number of bins and schedule their deliveries so that the inventory cost of each bin delivery is upper bounded by $B$.


An approach for solving the problem should specify a batching policy for grouping items into bins and a scheduling policy for deciding bin delivery times.
Our contribution is threefold: we propose algorithmic approaches for solving the problem using different batching and scheduling policies \cite{Hall2003,Letsios2021}, we analyze the approximability of the resulting algorithms and derive positive and negative bounds \cite{Christensen2017,Letsios2020}, we elaborate on the problem's hardness and discuss practical implications \cite{Garey1979}.
Compared to classic bin packing approaches, our results use new bounds and structural properties exploiting the temporal nature of container delivery scheduling.
Combining transportation and inventory costs is the key technical difficulty in the analysis \cite{Chitsaz2019,Elimam2013}.

\paragraph{Related Work}

The container delivery scheduling problem is related to the batch delivery scheduling problem introduced by \cite{Hall2003}.
Both problems include a single supplier receiving orders by multiple customers, that have to be delivered in batches with the aim of optimizing the number of shipments required for the distribution and the total inventory cost.
However, the former models inventory management in the distribution side, while the latter models inventory management in the production side. 
In particular, we investigate packing products to batch deliveries and explore trade-offs between the number of shipments and inventory storage/backlog costs at distribution centers, while \cite{Hall2003} focus on sequencing production jobs in a single machine and storing produced items in the manufacturing site before distribution to increase batch sizes during shipping. 
The batch delivery scheduling problem is strongly $\mathcal{NP}$-hard and admits a greedy 3/2-approximation algorithm \cite{Selvarajah2009}.

The container delivery scheduling problem can be viewed as parallel machine scheduling with the objective of minimizing the total earliness and tardiness \cite{Baker1990}.
Feasible solutions to the two problems calculate inventory costs of ordered items before these are delivered to the corresponding buyers in the same way.
However, there are two key differences.
First, our objective is to minimize the number of deliveries, whereas earliness tardiness scheduling assumes a fixed number of deliveries.
Second, we incorporate individual bounds on the inventory cost of different deliveries, whereas earliness tardiness scheduling minimizes the total inventory cost as a whole.
This discrepancy in the objective function and constraints results in algorithms and bounds for the container delivery scheduling problem that are of different form compared to ones for earliness tardiness scheduling.
The earliness tardiness scheduling problem is $\mathcal{NP}$-hard even in the case of a single machine (i.e.\ bin in our context) and admits a fully-polynomial time approximation scheme \cite{Kovalyov1999}.
Bounds and structural properties applicable to the parallel machine version of the problem are given in \cite{Jabbari2021,Kedad2008,Rolim2020}.

Finally, the container delivery scheduling problem is a variant of temporal bin packing problems 
\cite{DellAmico2020}. 
Both problem types seek an assignment of items to bins and incorporate time by associating request/due times to items.
Nevertheless, our model captures supply chain batch shipping with bins used for a limited amount of time, whereas temporal bin packing is typically used for cloud computing applications and assumes bins available during an entire time horizon of interest \cite{Aydin2020}.
Also, we impose capacity constraints on the inventory costs incurred by bin deliveries, whereas temporal bin packing problems impose capacity contraints on the total sizes of items in bins.

\paragraph{Contributions and Paper Organization}


Despite the aforementioned commonalities with prior literature, to the authors knowledge, this work is the first to develop provably efficient algorithms for the proposed container delivery scheduling problem \cite{Christensen2017}.
The main paper contribution is the development and analysis of container delivery scheduling algorithms producing solutions by using different consolidation and scheduling policies.
Note that each delivered container is associated with a time interval between the earliest and latest due time of an item assigned to that container. 
Based on the structure of these time intervals, we distinguish sequential and nested consolidation policies.
The remainder of the manuscript proceeds as follows.

Section~\ref{Section:Preliminaries} defines the bin delivering scheduling problem, introduces notation, and shows that the problem is $\mathcal{NP}$-hard.
Section~\ref{Section:Sequential_Algorithms} elaborates on the approximability of algorithms greedily producing sequential solutions based on three different scheduling policies. 
Initially, we demonstrate that the algorithm achieves an arbitrarily large approximation ratio with only early deliveries (or only late deliveries).
However, we show that it becomes 4-approximate when it combines early and late deliveries.
Next, we prove that the algorithm attains an 8/3-approximation ratio when using median times for scheduling bin deliveries.
The analysis is based on (1) a time horizon partitioning consisting of time intervals with consecutive bin deliveries in the algorithm's solution, (2) bounding and accumulating the inventory cost incurred by the items of every interval in the optimal solution by accounting for the number of deliveries taking place therein.



Section~\ref{Section:Decoupling_Algorithm} proposes a 3-approximation algorithm that produces nested solutions by decoupling delivery scheduling from batching decisions.
This decoupling algorithm consists of a dynamic programming component for computing delivery times resulting in a minimal total inventory cost and a classic bin packing component for assigning items to bin deliveries.
The algorithm has the advantage of computing solutions with low total inventory cost, but the limitation that the bin utilization might be low and result in high transportation costs. 
To overcome this limitation, we refine the scheduling decisions after running the algorithm.
This approach allows obtaining a refined decoupling algorithm which is tightly 2-approximate.
Section~\ref{Section:Conclusion} concludes with a brief discussion on our results and future directions.

\section{Preliminaries}
\label{Section:Preliminaries}


A problem instance $\langle \mathcal{I},B\rangle$ consists of a set $\mathcal{I}=\{1,\ldots,n\}$ of $n$ items and an inventory bound $B$.
Each item $i\in\mathcal{I}$ is associated with a due time $d_i\in\mathbb{Z}^+$.
Set $\delta_{i}(\tau)=|d_i-\tau|$, for each item $i\in\mathcal{I}$ and time $\tau\geq0$.
Also, set $\delta(\sigma,\tau)=|\sigma-\tau|$ equal to the length of the time interval between every pair of times $\sigma,\tau\geq0$.
Each item must assigned to/packed in exactly one bin delivery, i.e.\ a feasible solution $\mathcal{S}$ is a partitioning of the items into $m$ of subsets $\mathcal{S}_1,\ldots,\mathcal{S}_m\subseteq\mathcal{I}$, where $\mathcal{S}_j$ is the subset of items packed in bin $j$.
Further, we must decide a time at which bin $\mathcal{S}_j$ will be delivered at the pick-up location, for $j\in\{1,\ldots,m\}$. 
If bin $\mathcal{S}_j$ is delivered at time $\mu_j$, then the items in $\mathcal{S}_j$ are available for collection in the pick-up location by the respective customers at $\mu_j$.
In this case, the bin is associated with an inventory cost $\sum_{i\in\mathcal{S}_j}\delta_i(\mu_j)$, where $\delta_i(\mu_j)$ corresponds to a storage cost for each item such that $d_i>\mu_j$, and a backlog cost for any item $i$ with $d_i<\mu_j$. 
We denote the delivery time of item $i\in\mathcal{I}$ by $v_i$, that is $v_i=\mu_j$, if $i\in\mathcal{S}_j$.
The inventory cost incurred by a bin delivery must be upper bounded by $B$.
The objective is to pack the items into bins and schedule the bin deliveries so that the number $m$ of used bins is minimized, and the bin inventory constraints are satisfied, i.e.\ $\sum_{i\in\mathcal{S}_j}\delta_i(\mu_j)\leq B$ for every $j\in\{1,\ldots,m\}$.

Consider the time interval $[a_j,b_j]$ associated with bin $j$ in $\mathcal{S}$, where $a_j=\min_{i\in\mathcal{S}_j}\{d_i\}$ and $b_j=\max_{i\in\mathcal{S}_j}\{d_i\}$.
Based on the structure of these intervals, we distinguish sequential and nested solutions. 
Solution $\mathcal{S}$ is \emph{sequential} if there is an ordering of the bins s.t.\ $\mu_1\leq\ldots\leq\mu_m$ and $d_i\leq d_{i'}$, for each pair of items $i\in\mathcal{S}_j$ and $i'\in\mathcal{S}_{j'}$ with $1\leq j<j'\leq m$.
Solution $\mathcal{S}$ is \emph{nested}, if the bins can be numbered s.t.\ either $[a_j,b_j]\supseteq[a_{j'},a_{j'}]$ or $[a_j,b_j]\cap[a_{j'},b_{j'}]=\emptyset$, for each $1\leq j<j'\leq m$.
Sequential solutions are obtained when orders are processed on a first-come first-served basis.
Nested solutions enable more complex order priorities. 
The delivery time $\mu_j$ of bin $j\in\{1,\ldots,n\}$ is characterized as \emph{early}, \emph{median}, or \emph{late} based on whether it is close to $a_j$, the median of the $\{d_i:i\in\mathcal{S}_j\}$ values, and close to $b_j$, respectively.

Next, we present a key property of an optimal solutions and prove the problem's $\mathcal{NP}$-hardness.
Consider a feasible solution $\mathcal{S}$ with $m$ bins $\mathcal{S}_1,\ldots,\mathcal{S}_m$.
Let $\mathcal{S}_j^-=\{i:d_i\leq \mu_j,i\in\mathcal{S}_j\}$ and $\mathcal{S}_j^+=\{i:d_i>\mu_j,i\in\mathcal{S}_j\}$ be the subsets of items in $\mathcal{S}_j$ due not later than and after the delivery time $\mu_j$, respectively, for $j\in\{1,\ldots,m\}$.
Also, let $n_j^-=|\mathcal{S}_j^-|$ and $n_j^+=|\mathcal{S}_j^+|$ be the cardinalities of these sets.
Lemma~\ref{Lemma:Bin_Delivery_Time} shows how to compute the delivery time of a bin once we know its content in an optimal solution.
Theorem~\ref{Thm:NP-hardness} characterizes the problem's computational complexity.

\begin{lemma}
\label{Lemma:Bin_Delivery_Time}
For each instance $\langle \mathcal{I},B\rangle$ of the container delivery scheduling problem, there exists an optimal solution $\hat{\mathcal{S}}$ with $\hat{m}$ bins s.t.\ (a) $\hat{n}_j^-=\hat{n}_j^+$ and (b) $\hat{\mu}_j\in\{d_i:i\in\mathcal{I}\}$, for each $j\in\{1,\ldots,\hat{m}\}$.
\end{lemma}
\begin{proof}
Consider an optimal solution $\hat{\mathcal{S}}$ with $\hat{m}$ bins, and assume for contradiction that $\hat{n}_j^-\neq \hat{n}_j^+$, for some $j\in\{1,\ldots,\hat{m}\}$.
W.l.o.g.\ $\hat{n}_j^->\hat{n}_j^++1$.
The cases $\hat{n}_j^-=\hat{n}_j^++1$ and $\hat{n}_j^-<\hat{n}_j^+$ can be handled with similar arguments.
Among all optimal solutions, consider as $\hat{\mathcal{S}}$ one for which $\hat{\mu}_j$ is minimal.
For simplicity, suppose that the items in $\hat{\mathcal{S}}_j$ are numbered as $\{1,\ldots,n_j\}$ in non-decreasing order of due dates, that is $d_1\leq\ldots\leq d_{n_j}$, and that there are no identical due times, i.e.\ $|d_i-d_{i'}|>\eta$, for an infinitesimal $\eta>0$ and every pair of items $i,i'\in\mathcal{I}$ s.t.\ $i\neq i'$.
By definition, $\hat{\mathcal{S}}_j$ must have a delivery time $\hat{\mu}_j\in[d_{\hat{n}_j^-},d_{\hat{n}_j^-+1})$.
Now, consider a new solution $\mathcal{S}$ that packs the items as in $\hat{\mathcal{S}}$, i.e.\ $\mathcal{S}_{j'}=\hat{\mathcal{S}}_{j'}$ for each $j'\in\{1,\ldots,\hat{m}\}$, and has the same bin delivery times except bin $j$.
In particular, $\mu_{j'}=\hat{\mu}_{j'}$, for $j'\in\{1,\ldots,\hat{m}\}\setminus\{j\}$, and $\mu_j=\hat{\mu}_j-\epsilon$, where $\epsilon>0$ is a value s.t.\ $\mu_j=d_{n_{j}^--1}$ with item $n_j^-$ belonging to $\mathcal{S}_j^+$ in $\mathcal{S}$.
The total inventory cost incurred by the items in $\mathcal{S}_j$ is:
\begin{align*}
\sum_{i\in\mathcal{S}_j}|d_i-\mu_j| & = \sum_{i\in\mathcal{S}_j^-}(\mu_j-d_i) 
+ \sum_{i\in\mathcal{S}_j^+}(d_i-\mu_j) \\
& = \sum_{i\in\hat{\mathcal{S}}_j^-\setminus\{\hat{n}_j^-\}}(\mu_j-d_i) + \sum_{i\in\hat{\mathcal{S}}_j^+\cup\{\hat{n}_j^-\}}(d_i-\mu_j) \\
& = \sum_{i\in\hat{\mathcal{S}}_j^-}(\mu_j-d_i) + \sum_{i\in\hat{\mathcal{S}}_j^+} (d_i-\mu_j) + 2(d_{\hat{n}_j^-}-\mu_j) \\
& = \sum_{i\in\hat{\mathcal{S}}_j^-}(\hat{\mu}_j-d_i) + \sum_{i\in\hat{\mathcal{S}}_j^+}(d_i-\hat{\mu}_j) + (|\hat{\mathcal{S}}_j^+|-|\hat{\mathcal{S}}_j^-|)\epsilon + 2(d_{n_j^-}-\mu_j) \\
& \leq B,
\end{align*}
where the second equality uses the fact that $\mathcal{S}_j^-=\hat{\mathcal{S}}_j^-\setminus\{n_j^-\}$ and $\mathcal{S}_j^+=\hat{\mathcal{S}}_j^+\cup\{n_j^-\}$, the third equality is based on simple algebraic manipulation, the forth equality holds because $\mu_j=\hat{\mu}_j-\epsilon$, while the inequality is obtained by the fact that $\hat{\mathcal{S}}$ is feasible, i.e.\ $\sum_{i\in\hat{\mathcal{S}}_j}|d_i-\mu_j|\leq B$, and our assumptions $n_j^+-n_j^-\leq -2$ and $d_{n_j^-}-\mu_j\leq\epsilon$.
Thus, we obtain a contradiction on the fact that $\hat{S}$ assigns a minimal delivery time to bin $j$.
A similar argument implies that $\hat{\mu}_j\in\{d_i:i\in\mathcal{I}\}$, for each $j\in\{1,\ldots,\hat{m}\}$.
\end{proof}

\begin{theorem}
\label{Thm:NP-hardness}
Container delivery scheduling is $\mathcal{NP}$-hard.
\end{theorem}
\begin{proof}
We present a reduction from 3-Partition which is known to be $\mathcal{NP}$-hard even with polynomially bounded parameters \cite{Garey1979}.
A 3-Partition instance $\langle\mathcal{A},\beta\rangle$ consists of a set $\mathcal{A}=\{a_1,\ldots,a_{3m}\}$ of $3m$ positive integers and an integer $\beta$ such that $\sum_{i\in\mathcal{A}}a_i=m\beta$ and $\beta/4\leq a_i\leq \beta/2$, for $i\in\mathcal{A}$.
We consider the case $\beta=p(m)$, where $p(m)$ is a polynomial of $m$.
The objective is to partition $\mathcal{A}$ into $m$ subsets $\mathcal{A}_1,\ldots,\mathcal{A}_m$ such that $\sum_{i\in \mathcal{A}_j}a_i=\beta$, for each $j\in\{1,\ldots,m\}$.
Starting from a 3-Partition instance, we construct a container delivery scheduling instance with $n=3m+m^3\beta^4$ items as follows. 
We set $\alpha=\max_{i\in\mathcal{A}}\{a_i\}$, assume w.l.o.g.\ that $a_1\geq a_2\geq\ldots\geq a_{3m}$, and select $d=\alpha\cdot\beta$.
We introduce $3m$ items with $d_i=d-a_i\cdot{\beta}$, for $i\in\{1,\ldots,3m\}$.
Further, we add $m^2\beta^4$ items with due time at $\tau_j=d+(\beta/3)+j(1/m\beta)$, for each $j\in\{1,\ldots,m\}$.
Figure~\ref{Fig:NP_Hardness_Construction} illustrates the constructed instance.
We argue that $\mathcal{A}$ admits a 3-Partition iff there exists a feasible solution to the container delivery scheduling instance $(\mathcal{I},B)$ with $m$ bins, where $\mathcal{I}=\mathcal{A}$ and $B=\beta^2+\beta+3$.

\begin{figure}[t]
\begin{center}
\includegraphics[scale=0.55]{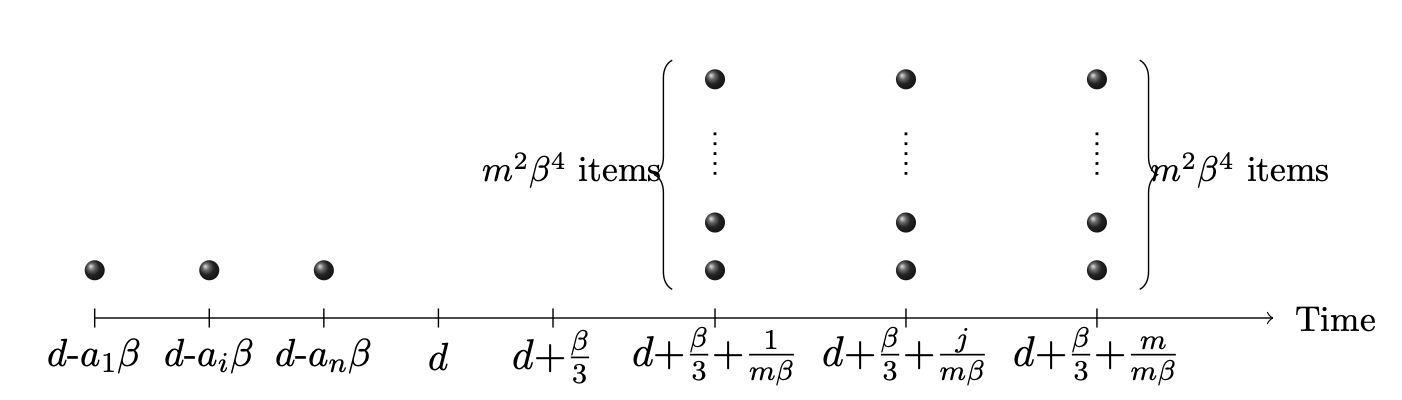}
\end{center}
\vspace*{-0.5cm}
\caption{NP-hardness construction. Each item is represented as a ball with a due time.}
\label{Fig:NP_Hardness_Construction}
\end{figure}


$\Longrightarrow$ :
Assume that $\mathcal{A}$ admits a 3-Partition $\mathcal{A}_1,\ldots,\mathcal{A}_m$.
Note that $|\mathcal{A}_j|=3$, for each $j\in\{1,\ldots,m\}$.
We build a feasible solution to the container delivery scheduling instance $\langle\mathcal{I},B\rangle$ with $m$ bins.
Bin $j\in\{1,\ldots,m\}$ contains the items in $\mathcal{A}_j$ that are due before $d$ and $m^2\beta^4$ items with an identical due time after $d$, i.e.\ $\mathcal{S}_j=\{i:i\in \mathcal{A}_j\}\cup\{i:d_i=\tau_j\}$, for $j\in\{1,\ldots,m\}$.
In addition, bin $j$ has delivery time $\mu_j=\tau_j$. 
The constructed solution is feasible because, for each $j\in\{1,\ldots,m\}$, we have that:
\begin{align*}
\sum_{i\in\mathcal{S}_j}|d_i-\mu_j| 
& = \sum_{i\in \mathcal{A}_j}[d+(\beta/3)+j(1/m\beta)-(d-a_i\beta)] \\
& \leq \sum_{i\in \mathcal{A}_j}a_i\beta + |\mathcal{A}_j|(\beta/3+1/\beta) \\
& \leq \beta^2 + \beta + 3. 
\end{align*}

$\Longleftarrow$ :
Suppose that there exists a feasible solution $\mathcal{S}$ for the container delivery scheduling instance 
$(\mathcal{I},B)$  with $m$ bins.
We divide the set of distinct delivery times into $\mathcal{R}^-=\{d-a_1\beta,\ldots,d-a_{3m}\beta\}$ and $\mathcal{R}^+=\{d+(\beta/3)+1/m\beta,\ldots,d+(\beta/3)+1/\beta\}$.
By Lemma~\ref{Lemma:Bin_Delivery_Time}, each bin delivery takes place at a time point in  
$\mathcal{R}=\mathcal{R}^-\cup\mathcal{R}^+$ in $\mathcal{S}$.

First, we claim that bin deliveries may only occur after time $d$, i.e.\ $\mu_j\in\mathcal{R}^+$ for each $j\in\{1,\ldots,m\}$, in $\mathcal{S}$.
Assume for contradiction that there exists one among the $m$ bin deliveries in $\mathcal{S}$ takes place before $d$.
Then, there also exists is a time point $\tau_j=d+(\beta/3)+j(1/m\beta)$, $j\in\{1,\ldots,m\}$, at which no delivery occurs.
Since $m^2\beta^4$ items are due at $\tau_j$, some bin $j'\in\{1,\ldots,m\}$ in $\mathcal{S}$ contains at least $m\beta^4$ of those items.
Clearly, $|\mu_{j'}-\tau_j|\geq 1/m\beta$.
So, it must be the case that $\sum_{i\in\mathcal{S}_{j'}}|d_i-\mu_{j'}|\geq m\beta^4(1/m\beta)>\beta^2+\beta+3$, which contradicts that $\mathcal{S}$ is feasible.
Thus, $\mathcal{S}$ uses $m$ bins each of which is delivered at one of the $m$ points in $\mathcal{R}^+$.

Next, let $\mathcal{A}_j=\{i:d_i\leq d,i\in\mathcal{S}_j\}$ be the subset of items in bin $j$ of $\mathcal{S}$ which have been constructed from integers in the 3-Partition instance $\langle\mathcal{A},\beta\rangle$.
We claim that $\mathcal{S}_j$ contains exactly 3 items originally obtained from $\mathcal{A}$, i.e.\ $|\mathcal{A}_j|=3$, for each $j\in\{1,\ldots,m\}$.
If this is not the case, then there exists a bin $j\in\{1,\ldots,m\}$ with $\geq4$ items.
By taking into account the previous claim, we have that $\mu_j-d_i\geq (d+\beta/3)-(d-\alpha_i\beta)\geq \beta^2/4+\beta/3$, for each $i\in\mathcal{A}_j$.
Hence, $\sum_{i\in\mathcal{S}_j}|d_i-\mu_j|\geq 4(\beta^2/4+\beta/3)>\beta^2+\beta+3$, a contradiction.

Finally, we claim that the original 3-Partition items $\mathcal{A}_j$ in bin $j$ sum to $\beta$, i.e.\ $\sum_{i\in\mathcal{A}_j}a_i=\beta$ for each $j\in\{1,\ldots,m\}$.
If not, then there exists a bin $j\in\{1,\ldots,m\}$ such that $\sum_{i\in\mathcal{A}_j}a_i\geq\beta+1$.
For this bin, we get that:
\begin{align*}
\sum_{i\in\mathcal{S}_j}|d_i-\mu_j|
& \geq \sum_{i\in\mathcal{A}_j}[d+(\beta/3)-(d-\beta a_i)] \\
& = \sum_{i\in\mathcal{A}_j}a_i\beta + |\mathcal{A}_j|(\beta/3) \\
& \geq (\beta+1)\beta + 3(\beta/3) > \beta^2+\beta+3
\end{align*}
We conclude that $\mathcal{A}$ admits a 3-Partition.
\end{proof}


\section{Greedy Sequential Algorithms}
\label{Section:Sequential_Algorithms}

This section elaborates on the approximability of algorithms producing sequential solutions.
Section~\ref{Section:Sequential_Algorithm_Description} 
describes sequential algorithms
using a greedy first-fit policy for assigning items to bins and different delivery scheduling policies.
On the negative side, Section~\ref{Section:Early_Deliveries} shows that 
an early delivery scheduling policy results in an arbitrarily bad approximation ratio.
On the positive side, Section~\ref{Section:Early_Late_Deliveries} demonstrates that a scheduling policy combining early and late deliveries
allows computing asymptotically 4-approximate solutions.
Section~\ref{Section:Median_Deliveries} derives an asymptotically 8/3-approximation algorithm by using a median-time 
delivery scheduling policy.

\subsection{Algorithm Description}
\label{Section:Sequential_Algorithm_Description}

Initially, the algorithm sorts the items in non-decreasing due times $d_1\leq\ldots\leq d_n$.
Next, it opens a bin $j=1$ and packs the item $i=1$ in $\mathcal{S}_j$.
Follow-up items are greedily packed in $j$ as long as the inventory constraint $\sum_{i\in\mathcal{S}_j}|d_i-\mu_j|\leq B$ is not violated.
If an item $i$ does not fit in $j$, then we open a new bin $j+1$ and proceed in the same fashion.
Determining whether a subset $\mathcal{S}$ of items fits in the same bin requires a scheduling policy $\Delta_{\mathcal{S}}$ that computes the delivery time of the items in $\mathcal{S}$ if exactly those are packed in one bin, and returns the inventory cost incurred by such a delivery.

\begin{algorithm}[h] \nonumber
\caption{Sequential Algorithm}
\begin{algorithmic}[1]
\STATE Sort item s.t.\ $d_1\leq\ldots\leq d_n$.
\STATE Open empty bin $j=1$ with $\mathcal{S}_j=\emptyset$.
\FOR {$i\in\{1,\ldots,n\}$}
\IF {$\Delta_{\mathcal{S}_j\cup\{i\}}> B$}
\STATE Open bin $j=j+1$ and set $\mathcal{S}_j=\emptyset$
\ENDIF
\STATE $\mathcal{S}_j=\mathcal{S}_j\cup\{i\}$
\ENDFOR
\end{algorithmic}
\label{Alg:Sequential}
\end{algorithm}

\subsection{Early Delivery Scheduling}
\label{Section:Early_Deliveries}

Given a subset $\mathcal{S}_j\subseteq\mathcal{I}$ of items packed in the same bin, the early delivery policy schedules the items in $\mathcal{S}_j$ to be delivered at time $\mu_j=\min_{i\in\mathcal{S}_j}\{d_i\}$. 
Theorem~\ref{Thm:ERTF} shows that the sequential algorithm achieves a poor approximation ratio with this scheduling policy.
In particular, we show the existence of a problem instance such that the algorithm's solution requires $k$ bin deliveries, where $k$ can be arbitrarily large, while all items are packed in a single bin in the optimal solution.




\begin{theorem}
\label{Thm:ERTF}
The sequential algorithm with the early delivery scheduling policy is $\omega(1)$-approximate.
\end{theorem}
\begin{proof}
Consider an instance with $\ell$ distinct item due times $\tau_1,\tau_2,\ldots,\tau_{\ell}$.
Let $n_t$ be the number of items due at time $\tau_t$, for $t\in\{1,2,\ldots,\ell\}$.
We select the $n_t$ values so that: $n_1=1$, $n_2=2$, and $n_t=4^{t-1}(\sum_{s=1}^{t-1}n_s)$, for $t\geq 3$.
Let $\gamma_t=\tau_{t+1}-\tau_t$, for each $t\in\{1,\ldots,\ell-1\}$.
We select the $\tau_t$ values so that $\tau_0=0$ and $\gamma_t=(B/n_{t+1})+1$, for $t\in\{1,\ldots,\ell-1\}$.
That is, we set $\tau_1=0$ and $\tau_t=\sum_{s=1}^{t-1}\gamma_s=B(\sum_{s=2}^t1/n_s)+t$, for $t\in\{2,\ldots,n\}$.
We consider a bin inventory bound $B\geq n^2\ell$.

Initially, Algorithm~\ref{Alg:Sequential} opens bin $j=1$ and packs the item due at $\tau_1=0$ in $j$.
Due to the early delivery policy, the delivery time of this bin is $\mu_1=0$.
Next, Algorithm~\ref{Alg:Sequential} considers one of the two items due at $\tau_2=(B/2+1)$ and packs it in bin $j=1$.
Clearly, the second item due at $\tau_2$ does not fit in bin $j=1$, since such a bin delivery would incur an inventory cost $(B/2+1)+(B/2+1)=B+2>B$.
Hence, Algorithm~\ref{Alg:Sequential} opens a new bin $j=2$ and packs the second item due at $\tau_2$ therein.
That is, $\mu_2=\tau_2$.
In the same spirit, Algorithm~\ref{Alg:Sequential} opens a new bin when considering the last item due at $\tau_t$, for each $t\in\{2,\ldots,\ell\}$.
Once this bin $j=t$ opens, Algorithm~\ref{Alg:Sequential} sets $\mu_j=\tau_t$ greedily packs $n_{t+1}-1$ items due at $\tau_{t+1}$.
Observe that not all $n_{t+1}$ items due at $\tau_{t+1}$ fit in a bin opened at $\tau_t$ since such a packing would incur an inventory cost $n_{t+1}\cdot \gamma_t=n_{t+1}[(B/n_{t+1})+1]=B+n_{t+1}>B$.
Nevertheless, $n_{t+1}-1$ items due at $\tau_{t+1}$ fit in a bin delivered at $\tau_t$ since $(n_{t+1}-1)[(B/n_{t+1})+1]\leq B -(B/n_{t+1}) + n_{t+1}\leq B$, where the last inequality holds because $B\geq n^2\geq n_{t+1}^2$.
Thus, the algorithm uses $\ell$ bins. 



On the other hand, we argue that there exists a feasible solution $\hat{\mathcal{S}}$ which packs all items in a single bin delivered at time $\hat{\mu}_1=\tau_{\ell}$.
The total inventory cost of such a bin delivery can be expressed as follows:

\begin{align}
\sum_{i\in\mathcal{I}}(\hat{\mu}_1-d_i) & = \sum_{t=1}^{\ell-1}\left(n_t\sum_{s=t}^{\ell-1}\gamma_s\right) \label{Eq:Omega_Interval_Structure} \\
& = n_1\cdot\gamma_1+\sum_{t=2}^{\ell-1} \left[\left(\sum_{s=1}^{t}n_s\right)\gamma_t\right] \label{Eq:Omega_Sum_Rearrangement} \\
& = \left(\frac{B}{2}+1\right)+\sum_{t=2}^{\ell-1} \left[\left( \sum_{s=1}^{t}n_s \right)\left(\frac{B}{4^t(\sum_{s=1}^{t}n_s)}+1\right)\right] \label{Eq:Omega_Gamma_Definition} \\
& \leq \frac{B}{2} + \sum_{t=1}^{\ell} \frac{B}{4^t} + \ell\cdot n \label{Eq:Omega_Constructed_Sequence} \\
& \leq \frac{B}{2} + \frac{B}{3} + \frac{B}{n}\leq B \label{Eq:Omega_Geometric_Series}.
\end{align}
Eq.~(\ref{Eq:Omega_Interval_Structure}) holds because the inventory cost by the $n_t$ items due at $\tau_t$ is $n_t(\sum_{s=t}^{\ell-1}\gamma_s)$, for $t\in\{1,\ldots,\ell-1\}$. 
Eq.~(\ref{Eq:Omega_Sum_Rearrangement}) is obtained by rearranging the summations. 
Eq.~(\ref{Eq:Omega_Gamma_Definition}) is based on the fact that $n_1=1$, $n_2=2$, 
i.e.\ $\gamma_1=(B/n_2)+1=(B/2)+1$, 
$\gamma_t=(B/n_{t+1})+1$ and $n_{t+1}=4^t\sum_{s=1}^{t}n_s$, for $t\in\{2,\ldots,\ell-1\}$. 
Eq.~(\ref{Eq:Omega_Constructed_Sequence}) cancels out the terms $\sum_{s=1}^t n_s$
and uses the equality $\sum_{t=1}^{\ell}n_t= n$.
Eq.~(\ref{Eq:Omega_Geometric_Series}) is obtained with the standard geometric series sum $\sum_{t=1}^{\ell}(1/4^t)=(1/4)(1-(1/4^{\ell}))/(1-1/4)\leq1/3$ and the inequality $B\geq n^2\ell$. 
\end{proof}

\subsection{Combined Early and Late Delivery Scheduling}
\label{Section:Early_Late_Deliveries}

Assume that the bins are numbered s.t.\ $\mu_1\leq\ldots\leq\mu_m$ in any feasible solution $\mathcal{S}$.
Given a subset $\mathcal{S}_j\subseteq\mathcal{I}$ of items packed in the same bin, a policy combining early and late deliveries schedules the items in $\mathcal{S}_j$ to be delivered at time $\mu_j=\min_{i\in\mathcal{S}_j}\{d_i\}$, if $j$ is odd, and at time $\mu_j=\max_{i\in\mathcal{S}_j}\{d_i\}$, if $j$ is even. 
Theorem~\ref{Theorem:Early_Late_Deliveries} shows that the sequential algorithm is asymptotically 4-approximate with this scheduling policy.

Consider the solution $\mathcal{S}$ derived by Algorithm~\ref{Alg:Sequential} with the combined early and late delivery scheduling policy and set $k=\lfloor m/2\rfloor$, where $m$ is the number of bin deliveries in $\mathcal{S}$. 
Let $\mathcal{T}=\{1,\ldots,k\}$ be the set of time intervals $[a_1,b_1],\ldots,[a_k,b_k]$, where $a_t$ and $b_t$ correspond to the minimum and maximum due time in an odd and an even bin, respectively, in $\mathcal{S}$, for each $t\in\mathcal{T}$. 
That is, we set $a_t=\min_{i\in\mathcal{S}_{2t-1}}\{d_i\}$ and $b_t=\max_{i\in\mathcal{S}_{2t}}\{d_i\}$, for $t\in\mathcal{T}$.
Clearly, Algorithm~\ref{Alg:Sequential} performs $m\leq 2k+1$ bin deliveries, with exactly one delivery at every $a_t$ and $b_t$ time point. 

Due to the sequential nature of $\mathcal{S}$, it holds that $b_s\leq a_t$, for each $s,t\in\mathcal{T}$ s.t.\ $s<t$.
Observe that the union $\cup_{t=1}^k\{[a_t,b_t]\}$ of the time intervals of interest does not span the entire time horizon $[\min_{i\in\mathcal{I}}\{d_i\},\max_{i\in\mathcal{I}}\{d_i\}]$. 
In particular, the right endpoint $b_t$ of interval $t\in\mathcal{T}\setminus\{k\}$ will not concide with the left endpoint $a_{t+1}$ of the follow-up interval $t+1$. 
However, item due times may only appear inside these intervals, i.e.\ $d_i\in\cup_{t=1}^k\{[a_t,b_t]\}$, for each $i\in\mathcal{I}$.
For each $t\in\mathcal{T}$, let $\mathcal{I}_t=\mathcal{S}_{2t-1}\cup\mathcal{S}_{2t}$ be the subset of items packed in bins $2t-1$ and $2t$, i.e.\ the items with $d_i\in[a_t,b_{t}]$. 

Next, consider an optimal solution $\hat{\mathcal{S}}$ and let $\hat{m}_t$ equal to the number of bin deliveries occuring during $t\in\mathcal{T}$ in $\hat{\mathcal{S}}$.
Then, $\mathcal{T}$ can be partitioned into the subsets $\mathcal{X}=\{t:\hat{m}_t=0,t\in\mathcal{T}\}$ and  $\mathcal{Y}=\{t:\hat{m}_t\geq1,t\in\mathcal{T}\}$ of intervals without and with at least one bin deliveries, respectively.
Suppose that $\hat{v}_i$ is the delivery time of item $i\in\mathcal{I}$ in $\hat{S}$.
The term $\hat{\lambda}_t=\sum_{i\in\mathcal{I}_t}\delta_i(\hat{v}_i)$ corresponds to the total inventory cost incurred in $\hat{\mathcal{S}}$ by the items due during $t\in\mathcal{T}$.
Lemma~\ref{Lemma:Single_Interval_Bounds} lower bounds these terms for $\mathcal{X}$ intervals.

\begin{lemma}
\label{Lemma:Single_Interval_Bounds}
For each interval $t\in\mathcal{X}$, it holds that $\hat{\lambda}_t>B$.
\end{lemma}
\begin{proof}
Consider an arbitrary time interval $t\in\mathcal{X}$ such that $\hat{m}_t=0$ in $\hat{\mathcal{S}}$.
Let $\mathcal{A}_t=\mathcal{S}_{2t-1}$ and $\mathcal{B}_t=\mathcal{S}_{2t}$ be the subsets of items assigned to the two bins delivered at $a_t$ and $b_t$, respectively, in the algorithm's solution $\mathcal{S}$. 
Suppose that $i\in\mathcal{B}_t$ is the item that does not fit in the same bin delivery with all items in $\mathcal{A}_t$ and results in opening a new bin for $\mathcal{B}_t$ in $\mathcal{S}$.
Due to the combined early and late delivery scheduling policy, it must be the case that $\sum_{h\in\mathcal{A}_t}\delta_h(a_t)+\delta(a_t,d_i)>B$.
Similarly, because the item due at $a_{t+1}$ does not fit in the same bin delivery with the $\mathcal{B}_t$ items, we have that $\sum_{h\in\mathcal{B}_t}\delta_h(a_{t+1})>B$. 
If $\delta(a_t,d_i)\leq\delta(d_i,a_{t+1})$, then $\hat{\lambda}_t\geq\sum_{i\in\mathcal{A}_t\cup\{i\}}\delta_i(a_t)>B$.
Otherwise, $\hat{\lambda}_t\geq\sum_{h\in\mathcal{B}_t}\delta_h(a_{t+1})>B$.
\end{proof}

\begin{theorem}
\label{Theorem:Early_Late_Deliveries}
The sequential algorithm combining early and late deliveries is asymptotically 
$4$-approximate.
\end{theorem}
\begin{proof}
Denote by $x=|\mathcal{X}|$ and $y=|\mathcal{Y}|$ the cardinalities of the sets of intervals $\mathcal{X}$ and $\mathcal{Y}$.
Since $\hat{S}$ is a feasible solution, it must be the case that $\sum_{t\in\mathcal{X}}\hat{\lambda}_t\leq (y+1)B$.
Then, Lemma~\ref{Lemma:Single_Interval_Bounds} gives $x\leq y+1$.
Therefore, we get that $m\leq2(x+y)+1\leq 4y+3$.
Given that $\hat{m}\geq y$, the theorem follows.
\end{proof}

\subsection{Median-Time Scheduling}
\label{Section:Median_Deliveries}

Given a subset $\mathcal{S}_j\subseteq\mathcal{I}$ of items packed in the same bin, the median-time scheduling policy delivers the items in $\mathcal{S}_j$ at time $\mu_j=M(\mathcal{S}_j)$, where $M(\mathcal{S}_j)$ is the median of the set $\{d_i:i\in\mathcal{S}_j\}$.
Recall that, due to Lemma~\ref{Lemma:Bin_Delivery_Time}, there exists an optimal solution scheduling each bin delivery in this way.
The main technical difficulty in analyzing the algorithm's performance is combining bounds on total inventory cost and the number of bin deliveries in the $\hat{\mathcal{S}}$.
Section~\ref{Section:Early_Late_Deliveries} derives an analysis by quantifying the effect of including or not a bin delivery in each interval $t\in\mathcal{T}$.
Here, we obtain tighter bounds by analyzing these effects with pairs of consecutive intervals.
The resulting bounds are summarized in Table~1 
and proven with Lemmas~\ref{Lemma:Consecutive_Interval_Bounds_1}-\ref{Lemma:Consecutive_Interval_Bounds_3}. 
Theorem~\ref{Theorem:Median_Scheduling} shows that the sequential algorithm is asymptotically 8/3-approximate with the median-time scheduling policy.

Similarly to Section~\ref{Section:Early_Late_Deliveries}, given a solution $\mathcal{S}$ with $m$ bins produced by the sequential algorithm under the median-time scheduling policy, denote by $\mathcal{T}=\{1,\ldots,k\}$, where $k=\lfloor m/2\rfloor$, the time-horizon partitioning $[a_1,b_1],\ldots,[a_k,b_k]$ obtained by setting $a_t=\min_{i\in\mathcal{S}_{2t-1}}\{d_i\}$ and $b_t=\max_{i\in\mathcal{S}_{2t}}\{d_i\}$, for each $t\in\mathcal{T}$.
As in the Lemma~\ref{Lemma:Single_Interval_Bounds} proof, we denote by $\mathcal{A}_t$ and $\mathcal{B}_t$ the subsets of items packed in each of the two deliveries during $t\in\mathcal{T}$ in $\mathcal{S}$.
In addition, consider an optimal solution $\hat{\mathcal{S}}$, where $\hat{m}_t$ is equal to the number of bin deliveries occuring during $t\in\mathcal{T}$ in $\hat{\mathcal{S}}$.
Now, we partition $\mathcal{T}$ into the three subsets of intervals $\mathcal{X}=\{t:\hat{m}_t=0,t\in\mathcal{T}\}$ with no bin deliveries, $\mathcal{Y}=\{t:\hat{m}_t=1,t\in\mathcal{T}\}$ with exactly one bin delivery, and $\mathcal{Z}=\{t:\hat{m}_t\geq2,t\in\mathcal{T}\}$ with at least two bin deliveries, in $\hat{\mathcal{S}}$.
Let $\hat{\lambda}_t=\sum_{i\in\mathcal{I}_t}\delta_i(\hat{v}_i)$ be the total load incurred by the $\mathcal{I}_t$ items in  $\hat{\mathcal{S}}$.
Lemma~\ref{Proposition:Algorithm_Interval_Bounds} is a straightforward implication of the median-time scheduling and expresses inequalities causing the algorithm to open bins.
\begin{lemma}
\label{Proposition:Algorithm_Interval_Bounds}
Consider a time interval $t\in\mathcal{T}$ and let $i\in\mathcal{B}_t$ and $g\in\mathcal{A}_{t+1}$ be the items with the minimal due times in $\mathcal{B}_t$ and $\mathcal{A}_{t+1}$, respectively.
Then, it holds that: 
\begin{itemize}
    \item $\sum_{h\in\mathcal{A}_t}\delta_h(\tau)+\delta_i(\tau)>B$, for each $\tau\in[a_t,d_i]$,
    \item $\sum_{h\in\mathcal{B}_t}\delta_h(\tau)+\delta_g(\tau)>B$, for each $\tau\in[d_i,b_t]$. 
\end{itemize}
\end{lemma}

\begin{table}[!t]
\begin{center}
\footnotesize
\begin{tabular}{ |c|c|c| } 
 \hline
 Intervals & $\hat{\lambda}_t+\hat{\lambda}_{t+1}$ & $\hat{m}_t+\hat{m}_{t+1}$ \\ 
 \hline
 $t,t+1\in\mathcal{X}$ & $\geq3B$ & 0 \\ 
 $t\in\mathcal{X}$ (or $\mathcal{Y}$), $t+1\in\mathcal{Y}$ (resp.\ $\mathcal{X}$) & $\geq2B$ & 1 \\ 
 $t\in\mathcal{X}$ (or $\mathcal{Z}$), $t+1\in\mathcal{Z}$ (resp.\ $\mathcal{X}$) & $\geq2B$ & 2 \\ 
 $t,t+1\in\mathcal{Y}$ & $\geq B$ & 2 \\ 
 $t\in\mathcal{Y}$ (or $\mathcal{Z}$), $t+1\in\mathcal{Z}$ (resp.\ $\mathcal{Y}$) & $\geq 0$ & 3 \\ 
 \hline
\end{tabular}
\end{center}
\label{Table:Bounds}
\caption{Bounds on the total inventory cost and number of used bins for a pair of two consecutive intervals in an optimal solution.}
\end{table}

\begin{lemma}
\label{Lemma:Consecutive_Interval_Bounds_1}
For each pair of consecutive intervals $t,t+1\in\mathcal{X}$, it holds that $\hat{\lambda}_t+\hat{\lambda}_{t+1}\geq 3B$.
\end{lemma}
\begin{proof}
There are two cases: either $\delta(a_t,a_{t+1})<\delta(a_{t+1},b_{t+1})$, or not.
We only show the lemma for the former case, since the proof of the latter is quite similar (see Figure~\ref{Figure:Consecutive_Intervals}).
Consider the interval $[a_t,a_{t+1}]$.
Let $i\in\mathcal{B}_t$ and $g\in\mathcal{A}_{t+1}$ be the items that do not fit in the same bin delivery with all $\mathcal{A}_t$ and $\mathcal{B}_t$ items, thus lead the algorithm to open a new bin for the $\mathcal{B}_t$ and $\mathcal{A}_{t+1}$ items, respectively.
By Proposition~\ref{Proposition:Algorithm_Interval_Bounds}, we have that $\sum_{h\in\mathcal{A}_t}\delta_h(\sigma)+\delta_i(\sigma)>B$, for each $\sigma\in[a_t,d_i]$, and $\sum_{h\in\mathcal{B}_t}\delta_h(\tau)+\delta_g(\tau)>B$, for each $\tau\in[d_i,a_{t+1}]$.
Since $d_h\geq d_i$, for each $h\in\mathcal{B}_t$, we get that $\sum_{h\in\mathcal{B}_t}\delta_h(a_t)>B+\delta_i(a_t)$.
That is, $\sum_{i\in\mathcal{I}_t}\delta_h(a_t)\geq 2B-\delta_g(a_t)$.
Analogously, let $i'\in\mathcal{B}_{t+1}$ and $g'\in\mathcal{A}_{t+2}$ be the first items assigned to the $(2t+1)$-th and $2(t+1)$-th bin deliveries, respectively, in the solution $\mathcal{S}$ computed by the algorithm.
It must the case that $\sum_{h\in\mathcal{A}_{t+1}}\delta_h(\sigma)+\delta_{i'}(\sigma)>B$, for each $\sigma\in[a_{t+1},d_{i'}]$, and $\sum_{h\in\mathcal{B}_{t+1}}\delta_h(\tau)+\delta_{g'}(\tau)>B$, for each $\tau\in[d_{i'},a_{t+2}]$.
That is, either $\sum_{h\in\mathcal{A}_{t+1}\cup\{i'\}}\delta_h(a_{t+1})>B$, or $\sum_{h\in\mathcal{B}_{t+1}}\delta_h(a_{t+2})>B$.
Now, consider the optimal schedule $\hat{\mathcal{S}}$.
Because there is no bin delivery during $[a_t,b_{t+1}]$, an additional inventory cost of at least $\delta_g(a_t)$ is incurred by the $\mathcal{I}_{t+1}$ items, due to item $g$, in $\hat{\mathcal{S}}$.
Hence,
$\hat{\lambda}_{t}+\hat{\lambda}_{t+1}\geq\sum_{i\in\mathcal{I}_t}\delta_h(a_t)
+\min\left\{\sum_{h\in\mathcal{A}_{t+1}\cup\{i'\}}\delta_h(a_{t+1}),\sum_{h\in\mathcal{B}_{t+1}}\delta_h(a_{t+2})\right\}
+ \delta_g(a_t)
\geq 3B$. 
\end{proof}

\begin{figure*}
\begin{center}
\includegraphics[scale=0.65]{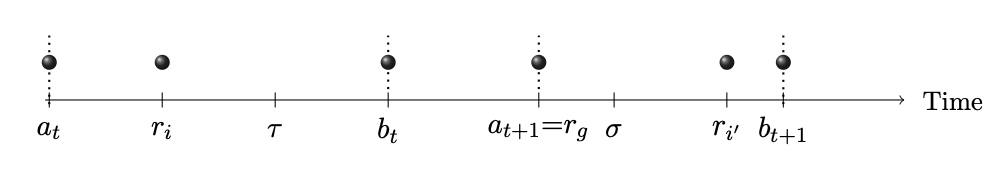}
\end{center}
\vspace*{-0.5cm}
\caption{Two consecutive intervals $t,t+1\in\mathcal{X}$ in the proof of Lemma~\ref{Lemma:Consecutive_Interval_Bounds_1}. Dotted lines correspond to interval boundaries.}
\label{Figure:Consecutive_Intervals}
\end{figure*}


\begin{lemma}
\label{Lemma:Consecutive_Interval_Bounds_2}
For each consecutive intervals $t\in\mathcal{X}$ and $t+1\in\mathcal{Y}$ (or $t\in\mathcal{Y}$ and $t+1\in\mathcal{X}$), it holds that $\hat{\lambda}_t+\hat{\lambda}_{t+1}\geq 2B$.
\end{lemma}
\begin{proof}
Consider a time interval $[a_t,b_{t+1}]$ obtained by merging $t\in\mathcal{X}$ and $t+1\in\mathcal{Y}$ (the case $t\in\mathcal{Y}$ and $t+1\in\mathcal{X}$ can be handled analogously), i.e.\ there is no bin delivery during $[a_t,b_t]$ and exactly one bin delivery occurs at time $\tau\in[a_{t+1},b_{t+1}]$ in $\hat{\mathcal{S}}$.
Denote by $i,g,i',g'\in\mathcal{I}$ the items with the earliest due times in $\mathcal{B}_t$, $\mathcal{A}_{t+1}$, $\mathcal{B}_{t+1}$, and $\mathcal{A}_{t+2}$, respectively.
Let $\mathcal{I}_t^+=\{h:\hat{v}_h\geq a_{t+1},h\in\mathcal{I}_t\}$ and $\mathcal{I}_{t+1}^-=\{h:\hat{v}_h\leq b_t,h\in\mathcal{I}_{t+1}\}$ be the $\mathcal{I}_t$ items delivered not earlier than $a_{t+1}$ and the $\mathcal{I}_{t+1}$ items delivered not later than $b_t$, respectively, in the optimal solution $\hat{\mathcal{S}}$. 
Then, either $|\mathcal{I}_t^+|<|\mathcal{I}_{t+1}^-|$, or $|\mathcal{I}_t^+|\geq|\mathcal{I}_{t+1}^-|$.

In the former case, we have that $\sum_{h\in\mathcal{I}_t^+}\delta_h(a_{t})+\delta(a_t,a_{t+1})\leq\sum_{h\in\mathcal{I}_{t+1}^-}\delta_h(a_t)$. 
By Proposition~\ref{Proposition:Algorithm_Interval_Bounds}, it holds that $\sum_{h\in\mathcal{A}_t}\delta_h(\sigma)+\delta_i(\sigma)>B$, for each $\sigma\in[a_t,d_i]$, and $\sum_{h\in\mathcal{B}_t}\delta_h(\sigma)+\delta_g(\sigma)>B$, for each $\sigma\in[d_i,a_{t+1}]$.
That is, $\sum_{h\in\mathcal{A}_t\cup\{i\}}\delta_h(a_t)>B$ and, given that $d_h\geq d_i$, for each $h\in\mathcal{B}_t$, we have $\sum_{h\in\mathcal{B}_t}\delta_h(a_t)+\delta(a_t,a_{t+1})\geq B+\delta_i(a_t)$.
So, $\sum_{h\in\mathcal{I}_t}\delta_h(a_t)+\delta(a_t,a_{t+1})\geq2B$.
Hence, the total inventory cost incurred in the algorithm's schedule by items due during $[a_t,b_{t+1}]$ is
$\hat{\lambda}_t +\hat{\lambda}_{t+1}\geq \sum_{h\in\mathcal{I}_t^-}\delta_h(a_t) + \sum_{h\in\mathcal{I}_{t+1}^-}\delta_h(a_{t}) 
\geq \sum_{h\in\mathcal{I}_t^-}\delta_h(a_t) + \sum_{h\in\mathcal{I}_{t}^+}\delta_h(a_t) + \delta(a_t,a_{t+1}) 
\geq 2B$.

Next, consider the case $|\mathcal{I}_t^+|\geq|\mathcal{I}_{t+1}^-|$.
Similarly to the Lemma~\ref{Lemma:Single_Interval_Bounds} proof, we have that $\hat{\lambda}_t\geq B$.
Let $\mathcal{A}_{t+1}^+=\{h:\hat{v}_h>\tau,h\in\mathcal{A}_{t+1}\}$ and $\mathcal{A}_{t+1}^-=\{h:\hat{v}_h\leq\tau,h\in\mathcal{A}_{t+1}\}$ be the $\mathcal{A}_{t+1}$ items delivered after $\tau$ and not later than $\tau$, respectively, in $\hat{\mathcal{S}}$.
Define $\mathcal{B}_{t+1}^-$ and $\mathcal{B}_{t+1}^+$ similarly.
If $|\mathcal{A}_{t+1}^+|\geq|\mathcal{B}_{t+1}^-|$, the fact that $\delta_h(a_{t+2})\geq\delta_{h'}(a_{t+2})$, for each $h\in\mathcal{A}_{t+1}$ and $h'\in\mathcal{B}_{t+1}$, gives that $\sum_{h\in\mathcal{A}_{t+1}^+}\delta_h(a_{t+2})\geq \sum_{h\in\mathcal{B}_{t+1}^-}\delta_h(a_{t+2})$.
But, Proposition~\ref{Proposition:Algorithm_Interval_Bounds} implies that $\sum_{h\in\mathcal{B}_{t+1}}\delta_h(a_{t+2})\geq B$.
Thus, $\lambda_{t+1}\geq\sum_{h\in\mathcal{A}_{t+1}^+}\delta_h(a_{t+2})+\sum_{h\in\mathcal{B}_{t+1}^+}\delta_h(a_{t+2})\geq\sum_{i\in\mathcal{B}_{t+1}}\delta_h(a_{t+2})\geq B$.
Next, suppose that $|\mathcal{A}_{t+1}^+|<|\mathcal{B}_{t+1}^-|$.
If $\mathcal{A}_{t+1}[\tau]$ is the subset of $\mathcal{A}_{t+1}$ items delivered at $\tau$, then we have that $\hat{\lambda}_{t+1}\geq\sum_{h\in\mathcal{I}_t^+}\delta_h(\tau)+\sum_{h\in\mathcal{A}_{t+1}[\tau]}\delta_h(\tau)+\sum_{h\in\mathcal{B}_{t+1}^-}\delta_h(\tau)\geq\sum_{h\in\mathcal{A}_t\cup\{i\}}\delta_h(\tau)\geq B$.
In all cases, we get $\hat{\lambda}_t+\hat{\lambda}_{t+1}>2B$.
\end{proof}


\begin{lemma}
\label{Lemma:Consecutive_Interval_Bounds_3}
For each pair of consecutive intervals $t,t+1\in\mathcal{Y}$, it holds that $\hat{\lambda}_t+\hat{\lambda}_{t+1}\geq B$.
\end{lemma}
\begin{proof}
If $\mathcal{I}_t^+=\{h:\hat{v}_h\geq a_{t+1},h\in\mathcal{I}_t\}$ are the $\mathcal{I}_t$ items delivered not earlier than $a_{t+1}$ and $\mathcal{I}_{t+1}^-=\{h:\hat{v}_h\leq b_t,h\in\mathcal{I}_{t+1}\}$ are the $\mathcal{I}_{t+1}$ items delivered not later than $b_t$, in the optimal solution $\hat{\mathcal{S}}$, then either $|\mathcal{I}_t^+|<|\mathcal{I}_{t+1}^-|$, or $|\mathcal{I}_t^+|\geq|\mathcal{I}_{t+1}^-|$.
Here, we only focus on the former case, since the proof for the latter is quite similar.




Suppose that the single delivery in interval $t$ occurs at time $\tau\in[a_t,b_t]$ in $\hat{\mathcal{S}}$.
Similarly to our previous proofs, consider a partitioning of the  $\mathcal{I}_t$ items into the subsets $\mathcal{A}_t$ and $\mathcal{B}_t$ corresponding to the contents of the two bin deliveries taking place in interval $t$ in the algorithm's schedule $\mathcal{S}$, and let $i$ be the first item assigned to the bin containing the $\mathcal{B}_t$ items.
We further split the $\mathcal{A}_t$ (and $\mathcal{B}_t$) sets into the subsets:
\begin{itemize}
    \item $\mathcal{A}_t[a_t]$ (resp.\ $\mathcal{B}_t[a_t]$) with delivery time $\leq a_t$ in $\hat{\mathcal{S}}$, 
    \item $\mathcal{A}_t[\tau]$ (resp.\ $\mathcal{B}_t[\tau]$) with delivery time at $\tau$ in $\hat{\mathcal{S}}$, 
    \item $\mathcal{A}_t[a_{t+1}]$ (resp.\ $\mathcal{B}_t[a_{t+1}]$) with delivery time $\geq a_{t+1}$ in $\hat{\mathcal{S}}$.
\end{itemize}

Distinguish two subcases based on whether $\tau\in[a_{t},d_i]$, or $\tau\in[d_i,b_t]$. 
In the former subcase, since $a_t\leq\tau\leq d_h$ for each $h\in\mathcal{B}_t[a_t]$, it must be the case that $\sum_{h\in\mathcal{B}_t[a_t]}\delta_h(a_t)\geq \sum_{h\in\mathcal{B}_t[a_t]}\delta_h(\tau)$.
Given that $\tau\leq d_h\leq a_{t+1}\leq d_{h'}$ for each $h\in\mathcal{B}_t$ and $h'\in\mathcal{I}_{t+1}^-$, and the fact that $|\mathcal{B}_t[a_{t+1}]|\leq|\mathcal{I}_t^+|<|\mathcal{I}_{t+1}^-|$, we have that $\sum_{h\in\mathcal{I}_{t+1}^-}\delta_h(\tau)\geq\sum_{h\in\mathcal{B}_t[a_{t+1}]}\delta_h(\tau) + \delta(\tau,a_{t+1})$.
Taking also into account Proposition~\ref{Proposition:Algorithm_Interval_Bounds},
\begin{align*}
\hat{\lambda}_t + \hat{\lambda}_{t+1}
& \geq \sum_{h\in\mathcal{B}_t[a_t]}\delta_h(a_t)+\sum_{h\in\mathcal{B}_t[\tau]}\delta_h(\tau) + \sum_{h\in\mathcal{I}_{t+1}^-}\delta_h(\tau) \\
& \geq \sum_{h\in\mathcal{B}_t[a_t]}\delta_h(\tau) + \sum_{h\in\mathcal{B}_t[\tau]}\delta_h(\tau) 
+ \sum_{h\in\mathcal{B}_t[a_{t+1}]}\delta_h(\tau) + \delta(\tau,a_{t+1}) \\
& = \sum_{h\in\mathcal{B}_t}\delta_h(\tau) + \delta(\tau,a_{t+1}) \\
& \geq B,
\end{align*}
In the latter subcase, $\tau\in[d_i,b_t]$.
Since $d_h\leq\tau\leq a_{t+1}$ for each $h\in\mathcal{A}_t$, it must be the case that $\sum_{h\in\mathcal{A}_{t}[a_{t+1}]}\delta_h(a_{t+1})\geq\sum_{h\in\mathcal{A}_{t}[a_{t+1}]}\delta_h(\tau)$.
Similarly to before, $\sum_{h\in\mathcal{I}_{t+1}^-}\delta_h(\tau)\geq\sum_{h\in\mathcal{B}_t[a_{t+1}]}\delta_h(\tau)$.
Thus,
\begin{align*}
\hat{\lambda}_t+\hat{\lambda}_{t+1}
& \geq \sum_{h\in\mathcal{A}_t[a_t]\cup\mathcal{B}_t[a_t]}\delta_h(a_t) 
+ \sum_{h\in\mathcal{A}_t[\tau]\cup\mathcal{B}_t[\tau]}\delta_h(\tau) 
+ \sum_{h\in\mathcal{A}_{t}[a_{t+1}]}\delta_h(a_{t+1})
+ \sum_{h\in\mathcal{I}_{t+1}^-}\delta_h(\tau) \\
& \geq \sum_{h\in\mathcal{A}_t[a_t]\cup\mathcal{B}_t[a_t]}\delta_h(a_t) 
+ \sum_{h\in\mathcal{A}_t[\tau]\cup\mathcal{B}_t[\tau]}\delta_h(\tau) 
\sum_{h\in\mathcal{A}_t[a_{t+1}]\cup\mathcal{B}_t[a_{t+1}]}\delta_h(\tau) + \delta(\tau,a_{t+1})
\end{align*}
If $|\mathcal{B}_t[a_t]|>|\mathcal{A}_t[\tau]|+|\mathcal{A}_t[a_{t+1}]|$, then the fact that $a_t\leq d_h\leq d_{h'}$, for each $h\in\mathcal{A}_t$ and $h'\in\mathcal{B}_t$, implies that
$\sum_{h\in\mathcal{B}_t[a_t]}\delta_h(a_t)\geq\sum_{i\in\mathcal{A}_t[\tau]\cup\mathcal{A}_t[a_{t+1}]}\delta_h(a_t)+\delta_i(a_t)$.
Therefore, 
\begin{align*}
\hat{\lambda}_t+\hat{\lambda}_{t+1}
& \geq\sum_{h\in\mathcal{A}_t[a_t]\cup\mathcal{B}_t[a_t]}\delta_h(a_t)\geq\sum_{h\in\mathcal{A}_t}\delta_h(a_t)+\delta_i(a_t) 
\geq B.
\end{align*}
If $|\mathcal{A}_t[\tau]|+|\mathcal{A}_t[a_{t+1}]|\geq|\mathcal{B}_t[a_t]|$, then we divide $\mathcal{B}_t[a_t]$ into the subsets $\mathcal{B}_t^-[a_{t}]$ and $\mathcal{B}_t^+[a_{t}]$ of items with due times before and not earlier than $\tau$, respectively.
Because $d_h\leq d_h'\leq\tau$, for each $h\in\mathcal{A}_t[\tau]\cup\mathcal{A}_t[a_{t+1}]$ and $h'\in\mathcal{B}_t^-[a_{t}]$, and $|\mathcal{A}_t[\tau]|+|\mathcal{A}_t[a_{t+1}]|\geq|\mathcal{B}_t^-[a_t]|$, we have that $\sum_{h\in\mathcal{A}[\tau]\cup\mathcal{A}[a_{t+1}]}\delta_h(\tau)\geq\sum_{h\in\mathcal{B}_t^-[a_{t}]}\delta_h(\tau)$.
Since $\tau\leq d_h$, for each $h\in\mathcal{B}_t^+[a_t]$, it must be the case that $\sum_{h\in\mathcal{B}_t^+[a_{t}]}\delta_h(a_{t})\geq \sum_{h\in\mathcal{B}_t^+[a_{t+1}]}\delta_h(\tau)$.
Similarly to before, since $|\mathcal{B}_t[a_{t+1}]|\leq|\mathcal{I}_{t+1}^-|$, it must be the case that $\sum_{h\in\mathcal{I}_{t+1}^-}\geq\sum_{h\in\mathcal{B}_t[a_{t+1}]}\delta_h(\tau)+\delta(\tau,a_{t+1})$. 
So, we have that: 
\begin{align*}
\hat{\lambda}_t+\hat{\lambda}_{t+1} 
& \geq \sum_{h\in\mathcal{A}_t[\tau]\cup\mathcal{A}_t[a_{t+1}]}\delta_h(\tau) 
+ \sum_{h\in\mathcal{B}_t^+[a_{t+1}]}\delta_h(a_{t+1}) 
+ \sum_{h\in\mathcal{B}_t[\tau]}\delta_h(\tau)
+ \sum_{h\in\mathcal{I}_{t+1}^{-} }\delta_h(\tau) \\
& \geq \sum_{h\in\mathcal{B}_t}\delta_h(\tau)+\delta(\tau,a_{t+1}) 
\geq B.
\end{align*}
In all cases, we conclude that $\hat{\lambda}_t+\hat{\lambda}_{t+1}\geq B$.
\end{proof}

\begin{theorem}
\label{Theorem:Median_Scheduling}
The sequential algorithm with the median-time scheduling policy is asymptotically $8/3$-approximate.
\end{theorem}
\begin{proof}
Consider the algorithm's solution $\mathcal{S}$ and an optimal solution $\hat{\mathcal{S}}$ with $m$ and $\hat{m}$ bin deliveries, respectively.  
To prove the theorem, we refine $\mathcal{T}$ by merging intervals.
W.l.o.g.\ we assume that $\mathcal{T}$ contains an even number of intervals.
In the new time horizon partitioning, which we denote by $\mathcal{T}'$, every odd interval $t\in\mathcal{T}$ is merged with the follow-up (even) interval $t+1\in\mathcal{T}$.
Let $\mathcal{X}\mathcal{X}$ be the subset of $\mathcal{T}'$ intervals obtained by merging two intervals $t,t+1\in\mathcal{X}$ of the original partitioning $\mathcal{T}$. 
Define $\mathcal{X}\mathcal{Y}$, $\mathcal{X}\mathcal{Z}$, $\mathcal{Y}\mathcal{Y}$, $\mathcal{Y}\mathcal{Z}$, and $\mathcal{Z}\mathcal{Z}$ analogously, by considering pairs of intervals in $\mathcal{X}$, $\mathcal{Y}$, and $\mathcal{Z}$. 
Let $xx$, $xy$, $xz$, $yy$, $yz$, and $zz$ be the cardinalities of these sets.
By definition, we have that $x=2xx+xy+xz$, $y=xy+2yy+yz$, and $z=xz+yz+2zz$.
Hence, $m \leq 2x+2y+2z+1 = 4(xx + xy + xz + yy + yz + zz)+1$.

Let $\hat{m}_t$ and $\hat{\lambda}_t$ be the number of bin deliveries and the total inventory cost incurred by the items with due times during an interval $t\in\mathcal{T}'$ of the refined partitioning in $\hat{\mathcal{S}}$.
For each $t\in\mathcal{X}\mathcal{X}$ interval, we have that $\hat{\lambda}_t\geq 3B$ and $\hat{m}_t=0$.
For each $t\in\mathcal{X}\mathcal{Y}$ interval, it holds that $\hat{\lambda}_t\geq 2B$ and $\hat{m}_t=1$.
For each $t\in\mathcal{X}\mathcal{Z}$ interval, $\hat{\lambda}_t\geq B$ and $\hat{m}_t\geq 2$. 
For each $t\in\mathcal{Y}\mathcal{Y}$ interval, it must be the case that $\hat{\lambda}_t\geq B$ and $\hat{m}_t=2$.
For each $t\in\mathcal{Y}\mathcal{Z}$ interval, $\hat{\lambda}_t\geq 0$ and $\hat{m}_t\geq 3$. 
For each $t\in\mathcal{Z}\mathcal{Z}$ interval, $\hat{\lambda}_t\geq 0$ and $\hat{m}_t\geq 3$. 
By considering the union 
of these intervals, observe that $\frac{1}{B}\sum_{t\in\mathcal{T}'}\hat{\lambda}_t+\sum_{t\in\mathcal{T}'}\hat{m}_t\geq 3(xx+xy+xz+yy+yz+zz)$.
Since the optimal solution $\hat{\mathcal{S}}$ is feasible, $\frac{1}{B}\sum_{t\in\mathcal{T}'}\hat{\lambda}_t \leq \hat{m}$.
Because $\hat{m}\geq \sum_{t\in\mathcal{T'}}\hat{m}_t$, we conclude that $m\leq(8/3)\hat{m}+1$.
\end{proof}

\paragraph{Example}
We complement our analysis with a 2 lower bound on the approximation ratio of the algorithm and leave closing the gap between this and the 8/3 upper bound as an open question.
We construct problem instance with two types of items $\mathcal{A}$ and $\mathcal{B}$, i.e.\ 
$\mathcal{I}=\mathcal{A}\cup\mathcal{B}$. 
Set $\mathcal{A}$ consists of $k$ subsets $\mathcal{S}_1^{\mathcal{A}},\ldots,\mathcal{S}_k^{\mathcal{A}}$ of items. 
Subset $\mathcal{S}_j^{\mathcal{A}}$ contains $n_j^{\mathcal{A}}=\ell$ items s.t.\ the $i$-th item of $\mathcal{S}_j^{\mathcal{A}}$ is due at time $\tau_{i,j}^{\mathcal{A}}=(j-1)\ell+i-1$, for $i\in\{1,\ldots,\ell\}$ and $j\in\{1,\ldots,k\}$.
Set $\mathcal{B}$ consists $k$ subsets $\mathcal{S}_1^{\mathcal{B}},\ldots,\mathcal{S}_k^{\mathcal{B}}$ of items, where $\mathcal{S}_j^{\mathcal{B}}$ contains $n_j^{\mathcal{B}}=(2B+1)$ items all due at time $\tau_j^{\mathcal{B}}=\ell k+j$.  
We select  $k=\sqrt{\lambda-1}$, $\ell=2\lambda$, $B=\lambda^2$, where $\lambda\geq1$ is a constant s.t.\ $\sqrt{\lambda-1}$ is an integer.
In the solution produced by the algorithm, each of the subsets $\mathcal{S}_j^{\mathcal{A}}$ and $\mathcal{S}_j^{\mathcal{B}}$ corresponds to the content of a single bin delivery, for $j\in\{1,\ldots,k\}$.
The median of the items in $\mathcal{S}_j^{\mathcal{A}}=\{(j-1)\ell,\ldots,j\ell-1\}$ is $\mu_j^{\mathcal{A}}=(j-1)\ell+\lambda-1$.
Given a $j\in\{1,\ldots,k\}$, the total cost incurred by the items in $\mathcal{S}_j^{\mathcal{A}}$ is $2(\sum_{h=1}^{\lambda-1}h)+\lambda=\lambda^2\leq B$.
Observe that the item with the earliest due time after the $\ell$-th $\mathcal{S}_j^{\mathcal{A}}$ item does not fit in the same bin delivery with all $\mathcal{S}_j^{\mathcal{A}}$ items.
Therefore, the algorithm produces a solution with $\ell=2k$ bin deliveries.
For a given $i\in\{1,\ldots,k\}$, let $\hat{\mathcal{S}}_i$ be the set of items containing the $i$-th item of $\mathcal{S}_j^{\mathcal{A}}$, for each $j\in\{1,\ldots,k\}$, and all items of $\mathcal{S}_i^{\mathcal{B}}$. 
In an optimal solution, the items in $\hat{\mathcal{S}}_i$ are simultaneously delivered at $\tau_i^{\mathcal{B}}=\ell k+i$, for each $i\in\{1,\ldots,k\}$.
With this bin delivery, the $i$-th item of the $t$-th $\mathcal{A}$ subset incurs a delivery cost $k\ell+i-(t-1)\ell-i=(k-t+1)$.
That is, the total inventory cost incurred by $\hat{\mathcal{S}}_i$ is $\sum_{h=1}^k(k-h+1)\ell=\sum_{h=1}^kh\ell=\frac{(k^2+k)\ell}{2}=\frac{((\lambda-1)^2+\lambda)(2\lambda)}{2}\leq B$.
Therefore, the solution is feasible and uses $k$ bins.

\section{Decoupling Algorithm}
\label{Section:Decoupling_Algorithm}

Next, we present an algorithm that may produce non-sequential solutions by decoupling delivery scheduling decisions from assignments of items to bins.
Section~\ref{Section:Decoupling_Description} describes the algorithm,
Section~\ref{Section:Decoupling_Dynamic_Programming} instantiates a dynamic programming component and Section~\ref{Section:Decoupling_Approximation_Ratio} derives the approximation ratio of the algorithm.
Section~\ref{Section:Decoupling_Refinement} proposes a way of refining the algorithm's schedule that allows computing asymptotically tight 2-approximate solutions.



\subsection{Algorithm Description.}
\label{Section:Decoupling_Description}

The proposed algorithm consists of $n$ iterations.
The $k$-th iteration produces a solution $\mathcal{S}^{(k)}$ with $k$ distinct delivery times, if such a feasible solution exists, for $k\in\{1,\ldots,n\}$.
For simplicity of presentation, consider the solution $\mathcal{S}$ computed for a fixed $k$ value.
Let $\tau_1<\ldots<\tau_k$ be the sequence of distinct delivery times and denote by $\mathcal{V}$ their indices in $\mathcal{S}$.
Note that multiple bin deliveries may occur at $\tau_t$, for each $t\in\{1,\ldots,k\}$.
Set $\Delta(\mathcal{V})=\sum_{i\in\mathcal{I}}\delta_i(\mathcal{V})$, where $\delta_i(\mathcal{V})=\min_{t\in\mathcal{V}}\{|d_i-\tau_t|\}$, for each $i\in\mathcal{I}$.
To obtain $\mathcal{V}$, the algorithm computes $k$ times such that the total length $\Delta(\mathcal{V})$ of the intervals between each item due time and every delivery time is minimized, and $\delta_i(\mathcal{V})\leq B$, for each $i\in\mathcal{I}$.
This computation, denoted by $\text{ET}(\mathcal{I},B,k)$, is a variant of classic parallel machine earliness-tardiness scheduling problems with an additional upper bound on the individual cost incurred by each job.
We show that $\text{ET}(\mathcal{I},B,k)$ can be computed in polynomial-time using dynamic programming.
Next, each item $i\in\mathcal{I}$ is assigned at its closest time $\arg\min_{t\in\mathcal{V}}\{d_i-\tau_t\}$ for delivery.
Let $\mathcal{D}_t\subseteq\mathcal{I}$ be the subset of items scheduled for delivery at $t\in\mathcal{V}$. 
Performing a single bin delivery at each $t\in\mathcal{V}$ might not result in a feasible solution, since it is possible that $\sum_{i\in\mathcal{D}_t}\delta_i(\mathcal{V})>B$.  
Hence, we use a classic bin packing algorithm $\text{BP}(\mathcal{D}_t,B)$, for each $t\in\mathcal{V}$, to pack the items in $\mathcal{D}_t$ into bins of capacity $B$, with each item $i\in\mathcal{D}_t$ having size $\delta_i(\mathcal{V})$.
A pseudocode for this procedure is given in Algorithm~\ref{Alg:Decomposition}.


\begin{algorithm}[h] \nonumber
\caption{Decoupling Algorithm}
\begin{algorithmic}[1]
\FOR {$k\in\{1,\ldots,n\}$}
\STATE $\mathcal{V}^{(k)}=\text{ET}(\mathcal{I},B,k)$
\STATE Set $\mathcal{D}_t^{(k)}=\emptyset$, for each $t\in\mathcal{V}^{(k)}$.
\FOR {$i\in\mathcal{I}$}
\STATE $\delta_i(\mathcal{V}^{(k)})=\arg\min_{t\in\mathcal{V}^{(k)}}\{|d_i-\tau_t|\}$
\STATE $\mathcal{D}_t^{(k)}=\mathcal{D}_t^{(k)}\cup\{i\}$
\ENDFOR
\STATE $\mathcal{S}^{(k)}=\emptyset$
\FOR {$t\in\mathcal{V}^{(k)}$}
\STATE $\mathcal{S}^{(k)}=\mathcal{S}^{(k)}\cup\text{BP}(\mathcal{D}_t^{(k)},B)$  
\ENDFOR
\ENDFOR
\STATE Return $\arg\min_{k\in\{1,\ldots,n\}}\{|\mathcal{S}^{(k)}|\}$
\end{algorithmic}
\label{Alg:Decomposition}
\end{algorithm}


\subsection{Delivery Time Computation}
\label{Section:Decoupling_Dynamic_Programming}

Given a container delivery scheduling instance $\langle\mathcal{I},B\rangle$, the algorithm's component $\text{ET}(\mathcal{I},B,k)$ seeks a set $\mathcal{V}$ of $k$ delivery times such that (a) the total inventory cost $\sum_{i\in\mathcal{I}}\delta_i(\mathcal{V})$ is minimized and (b) $\delta_i(\mathcal{V})\leq B$, for each $i\in\mathcal{I}$. 
We show that this computation can be performed in $O(n^3)$ time using dynamic programming.
An optimal solution $\mathcal{V}$ to $\text{ET}(\mathcal{I},B,k)$ defines a partitioning of $\mathcal{I}$ into $k$ subsets $\mathcal{D}_1,\ldots,\mathcal{D}_k$ such that $\mathcal{D}_t=\{i:\delta_i(\mathcal{V})=|d_i-\tau_t|,i\in\mathcal{I}\}$ contains all the items that closer to $\tau_t$ than any other time in $\mathcal{V}$, breaking ties arbitrarily.
Each subset $\mathcal{D}_t$ is associated with a time interval $[a_t,b_t]$, where $a_t=\min_{i\in\mathcal{D}_t}\{d_i\}$ and $b_t=\max_{i\in\mathcal{D}_t}\{d_i\}$.
By definition, these intervals are of sequential nature, i.e.\ $b_s\leq a_t$, for each pair of intervals $s,t\in\mathcal{V}$ such that $s<t$.
Because of the requirement $\delta_i(\mathcal{V})\leq B$, an arbitrary subset $\mathcal{V}$ of times may not be feasible for $\text{ET}(\mathcal{I},B,k)$.
The proposed algorithm computes feasible sequential solution as follows.

Number the items in $\mathcal{I}$ so that $d_1\leq\ldots\leq d_n$, breaking ties arbitrarily.
Consider a arbitrary solution $\mathcal{V}$ for $\text{ET}(\mathcal{I},B,k)$ and a subset $\{h,h+1,\ldots,i\}\subseteq\mathcal{I}$ of items.
Denote by $Q(h,i)$ the inventory cost incurred by the items in $\{h,h+1,\ldots,i\}$, if exactly those appear in a subset $\mathcal{D}_t$ of items assigned to a delivery time $t\in\mathcal{V}$.
Lemma~\ref{Lemma:Dynamic_Programming_Constants} shows how to compute $Q(h,i)$ for each $h,i\in\mathcal{I}$ such that $h\leq i$.
We make the convention that $Q(h,i)=+\infty$, if $\{h,h+1,\ldots,i\}$ cannot be feasibly assigned to the same delivery time.

\begin{lemma}
\label{Lemma:Dynamic_Programming_Constants}
Consider an arbitrary pair of items $h,i\in\mathcal{I}$ such that $h\leq i$.
If $d_i-d_h\leq2B$, then it holds that $Q(h,i)=\sum_{i'=h}^i|d_{i'}-\lambda(h,i)|$, where

\begin{align*}
\lambda(h,i) =
\begin{cases}
d_i-B, \quad \text{if $\med(h,i)\leq d_i-B$} \\
\med(h,i), \quad \text{if $d_i-B<\med(h,i)<d_h+B$} \\
d_h+B, \quad \text{if $\med(h,i)\geq d_h+B$}, \\
\end{cases} 
\end{align*}
where $\med(h,i)$ is the median of the set $\{d_h,d_{h+1},\ldots,d_i\}$. 
If $d_i-d_h>2B$, then $Q(h,i)=+\infty$.
\end{lemma}
\begin{proof}
Initially, observe that $Q(h,i)=+\infty$, for each pair of items $h,i\in\mathcal{I}$ such that $h<i$ and $d_i-d_h>2B$. 
Indeed, if $d_i-d_h>2B$, then $(d_i-\tau)+(\tau-d_h)>2B$, i.e.\ either $d_i-\tau>B$, or $\tau-d_h>B$, for each $\tau\in[d_h,d_i]$, which constradicts that the items $\{h,h+1,\ldots,i\}$ can assigned to the same bin in a feasible solution for $\text{ET}(\mathcal{I},B,k)$.
On the other hand, if $d_i-d_h\leq 2B$, set $\tau=(d_h+d_i)/2$.
For each $i'\in\{h,h+1,\ldots,i\}$, we distinguish two cases: either $d_{i'}\leq \tau$, or $d_{i'}>\tau$.
In the former case, $\tau-d_{i'}\leq(d_h+d_i)/2-d_h=(d_i-d_h)/2\leq B$.
In the latter case, $d_{i'}-\tau\leq d_i-(d_h+d_i/2)=(d_i-d_h)/2\leq B$.
In both cases, we conclude that there exists a time $\tau$ such that $|d_{i'}-\tau|\leq B$, for each $i'\in\{h,h+1,\ldots,i\}$.
Next, assume that $d_i-d_h\leq 2B$ and suppose that $\mathcal{D}_t=\{h,h+1,\ldots,i\}$, for some $t\in\mathcal{V}$ in a feasible solution $\mathcal{V}$ for $\text{ET}(\mathcal{I},B,k)$.
Clearly, $d_i-B\leq t\leq d_h+B$.
The corresponding $Q(i,h)$ values for this case can proved similarly to Lemma~\ref{Lemma:Bin_Delivery_Time}. 
If $d_i-B<\med(h,i)<d_h+B$, then the fact that $Q(h,i)=\sum_{i'=h}^i|d_{i'}-\lambda(h,i)|$ is a straightforward application of Lemma~\ref{Lemma:Bin_Delivery_Time}.
If $\med(h,i)\leq d_i-B$, a similar exchange argument to the one in Lemma~\ref{Lemma:Bin_Delivery_Time} implies that $\sum_{i'\in\{h,\ldots,i\}}|d_{i'}-(d_i-B)|\leq\sum_{i'\in\{h,\ldots,i\}}|d_{i'}-\tau|$, for each $\tau<d_i-B$.
The case $\med(h,i)\geq d_j+B$ can be handled analogously.
\end{proof}

To recursively decompose $\text{ET}(\mathcal{I},B,k)$, let $P(i,j)$ be cost of the subproblem of computing $j$ delivery times such that the  total distance between the due time of each item in $\{1,\ldots,i\}$ and its closest delivery time is minimized.
That is, 

\begin{align*}
P(i,j)={\arg\min}_{\mathcal{V}\subseteq\{d_1,\ldots,d_i\}}
\Bigg\{\sum_{h=1}^i\delta_h(\mathcal{V}): |\mathcal{V}|\leq j,\delta_h(\mathcal{V})\leq B, \forall h\in\{1,\ldots,i\}\Bigg\}
\end{align*}
If the subproblem is infeasible, then $P(i,j)=+\infty$.
Clearly, $P(n,k)$ is equivalent to $\text{ET}(\mathcal{I},B,k)$.
By definition, $P(i,1)=Q(1,i)$, for each $i\in\mathcal{I}$.
Also, $P(i,j)=0$, for each $i,j\in\{1,\ldots,n\}$ such that $i\leq j$.
Lemma~\ref{Lemma:Dynamic_Programming_Recurrence} specifies the recurrence relation of the dynamic programming algorithm.

\begin{lemma}
\label{Lemma:Dynamic_Programming_Recurrence}
For each $i,j\in\{2,\ldots,n\}$ s.t.\ $i>j$, it holds:

\begin{align*}
P(i,j) = \min_{1\leq h\leq i-1}\left\{P(h-1,j-1)+Q(h,i)\right\}
\end{align*}
\end{lemma}
\begin{proof}
We prove the lemma by induction on $j$.
If $j=1$, it clearly holds that $P(i,1)=Q(1,i)$, for each $i\in\{1,\ldots,n\}$.
Suppose that the lemma is true for $j-1$ delivery points and each item $i$ such that $i,j\in\{2,\ldots,n\}$ and $i>j$.
Then, consider an optimal solution $\mathcal{V}$ for $P(i,j)$.
Clearly, there exists an item $h\in\{2,\ldots,i\}$ such that $\cup_{t=1}^{j-1}\{\mathcal{D}_t\}=\{1,\ldots,h-1\}$ and $\mathcal{D}_j=\{h,\ldots,i\}$ in $\mathcal{V}$.
If $\{\tau_1,\ldots,\tau_{j-1}\}$ was not an optimal solution for $P(h-1,j-1)$ or $\sum_{i\in\mathcal{D}_j}|d_i-\tau_j|$ was not equal to $Q(h,i)$, we would obtain a contradiction that $\mathcal{V}$ is optimal for $P(i,j)$.
\end{proof}

Algorithm~\ref{Alg:Dynamic_Programming} provides a pseudocode for computing $ET(\mathcal{I},B,k)$, including the content of matrices $Q$, $P$, and the final solution.
The pseudcode contains three standard dynamic programming parts: (a) initialization, (b) recurrence, and (c) backtracking.
Lemma~\ref{Thm:Dynamic_Programming_Optimality} is an immediate corollary of Lemmas~\ref{Lemma:Dynamic_Programming_Constants}-\ref{Lemma:Dynamic_Programming_Recurrence}.

\begin{algorithm}[h] \nonumber
\caption{$\text{ET}(\mathcal{I},B,k)$}
\begin{algorithmic}[1]
\FOR {$h\in\{1,\ldots,n\}$}
\FOR {$i\in\{1,\ldots,n\}$}
\IF {$h\leq i$ and $d_i-d_h\leq 2B$}
\STATE $Q(h,i)=\sum_{i'=h}^i|d_{i'}-\lambda(h,i)|$
\ELSE
\STATE $Q(h,i)=+\infty$
\ENDIF
\ENDFOR
\ENDFOR 
\FOR {$j\in\{1,\ldots,n\}$}
\FOR {$i\in\{1,\ldots,n\}$}
\IF {$i\leq j$}
\STATE $P(i,j)=0$
\ENDIF
\IF {$j=1$}
\STATE $P(i,1)=Q(1,i)$
\ENDIF
\IF {$i>j$ and $j>1$}
\STATE $P(i,j)=\min_{h\in\{1,\ldots,i-1\}}\{P(h-1,j-1)+Q(h,i)\}$
\ENDIF
\ENDFOR
\ENDFOR
\STATE $i=n$; $j=k$
\WHILE {$j>1$}
\STATE $h=\arg\min_{h\in\{1,\ldots,n\}}\{P(h-1,j-1)+Q(h,j)\}$
\STATE $i=h-1$; $j=j-1$
\ENDWHILE
\end{algorithmic}
\label{Alg:Dynamic_Programming}
\end{algorithm}

\begin{lemma}
\label{Thm:Dynamic_Programming_Optimality}
$\text{ET}(\mathcal{I},B,k)$ can be computed in $O(n^3)$ time. 
\end{lemma}

\subsection{Approximability}
\label{Section:Decoupling_Approximation_Ratio}

Lemma~\ref{Lemma:Classic_Bin_Packing_Bound} is a classic bin packing property that we use for upper bounding the number of bins per delivery time in solutions produced by Algorithm~\ref{Alg:Decomposition}.
Theorem~\ref{Thm:Decoupling_Approximation_Ratio} establishes the algorithm's approximation ratio. 

\begin{lemma}
\label{Lemma:Classic_Bin_Packing_Bound}
Every collection $\{\delta_1,\ldots,\delta_n\}$ of $n$ integers, where $\delta_i\leq B$, for each $i\in\{1,\ldots,n\}$, can be partitioned into $m$ subsets $\mathcal{S}_1,\ldots,\mathcal{S}_{m}$ such that $m\leq 2\left(\frac{1}{B}\sum_{i=1}^{n}\delta_i\right)+1$.  
\end{lemma}
\begin{proof}
We claim that there exists a partitioning $\mathcal{S}_1,\ldots,\mathcal{S}_m$ with $m$ subcollections such that $\sum_{i\in\mathcal{S}_j\cup\mathcal{S}_{j'}}\delta_i\geq B$, for each $1\leq j<j'\leq m$.
Indeed, given a partitioning for which our claim is not true, we can merge $\mathcal{S}_j$ and $\mathcal{S}_{j'}$ and obtain a new partitioning with fewer subsets.
Starting from an arbitrary partitioning and repeating this merging operation, we conclude that our claim holds.
Assume w.l.o.g.\ that $m$ is an odd number, i.e.\ $m=2k+1$ for some integer $k>0$, in the resulting partitioning.
Then, we have that $\sum_{j=1}^{k}\sum_{i\in \mathcal{S}_{2j}\cup\mathcal{S}_{2j+1}}\delta_i\geq kB$.
Hence, $m\leq\frac{2}{B}\sum_{i=1}^n\delta_i+1$.
\end{proof}

\begin{theorem}
\label{Thm:Decoupling_Approximation_Ratio}
The decoupling algorithm is asymptotically 3-approximate.
\end{theorem}
\begin{proof}
Consider a container delivery scheduling instance $\langle\mathcal{I},B\rangle$ and an optimal solution $\hat{\mathcal{S}}$ with $\hat{m}$ bins for it.
We argue that the decoupling algorithm obtains a solution for $\langle\mathcal{I},B\rangle$ with $m$ bins, s.t.\ $m\leq2\hat{m}+1$.
For this, we elaborate on the solution $\mathcal{S}$ produced by the $k$-th iteration of the algorithm, where $k=\hat{m}$.
Let $\mathcal{V}=\{\tau_1,\ldots,\tau_k\}$ be the set of distict delivery times and denote by $m_t$ the number of bin deliveries at time $\tau_t$, for $t\in\mathcal{V}$, in $\mathcal{S}$.
Standard bin algorithms, e.g.\ First-Fit or First-Fit Decreasing, build solutions for $\text{BP}(\mathcal{D}_t^{(k)},B)$ satisfying Lemma~\ref{Lemma:Classic_Bin_Packing_Bound}.
So, $m_t\leq\frac{2}{B}\sum_{i\in\mathcal{D}_t}\delta_i(\mathcal{V})+1$. 
Therefore, we get
$m = \sum_{t\in\mathcal{V}}m_t \leq \frac{2}{B}\sum_{i\in\mathcal{I}}\delta_i(\mathcal{V}) + \hat{m}$.
Because each bin delivery satisfies the inventory bound in $\hat{\mathcal{S}}$, it must be the case that $\frac{1}{B}\sum_{i\in\hat{\mathcal{S}}_j}|d_i-\hat{\mu}_j|\leq 1$, for each $j\in\{1,\ldots,\hat{m}\}$.
Let $\hat{\mathcal{V}}$ be the set of distinct delivery times in $\hat{\mathcal{S}}$.
By distinguishing the delivery time $\hat{v}_i$ of each item $i\in\mathcal{I}$ in $\hat{\mathcal{S}}$ and the closest time in $\hat{\mathcal{V}}$ to $d_i$, we observe that
$\delta_i(\hat{\mathcal{V}})=\min_{t\in\mathcal{V}}\{|d_i-\hat{\tau}_t|\}\leq |d_i-\hat{v}_i|$.
A simple packing argument implies that
$\hat{m} \geq \frac{1}{B}\sum_{i\in\hat{\mathcal{S}}_j}|d_i-\hat{v}_i|\geq \frac{1}{B}\sum_{i\in\mathcal{I}}\delta_i(\hat{\mathcal{V}})$.
Because $\mathcal{V}$ is an optimal solution to $\text{ET}(\mathcal{I},B,k)$, it holds that $\Delta(\mathcal{V})\leq\Delta(\hat{\mathcal{V}})$, or equivalently $\sum_{i\in\mathcal{I}}\delta_i(\mathcal{V})\leq \sum_{i\in\mathcal{I}}\delta_i(\hat{\mathcal{V}})$.
The theorem follows.
\end{proof}

\subsection{Schedule Refinement}
\label{Section:Decoupling_Refinement}

The decoupling algorithm presented in Sections~\ref{Section:Decoupling_Description}-\ref{Section:Decoupling_Approximation_Ratio} computes a set $\mathcal{V}$ of $k$ distinct delivery times s.t.\ the total distance $\sum_{i\in\mathcal{I}}\delta_i(\mathcal{V})$ between the due time of each item $i\in\mathcal{I}$ and the closest delivery time in $\mathcal{V}$ is minimal. 
Suppose that $\mathcal{V}=\{\tau_1,\ldots,\tau_k\}$, where $\tau_1<\ldots<\tau_k$.
The time horizon can be partitioned into a set $\mathcal{T}=\{1,\ldots,k\}$ of $k$ intervals $[a_1,b_1],\ldots,[a_k,b_k]$, s.t.\ $[a_t,b_t]$ contains $\tau_t$ and the due times of the items in $\mathcal{I}_t=\{i:\arg\min_{\tau\in\mathcal{V}}\{\delta_i(\tau)\}=\tau_t,i\in\mathcal{I}\}$, for each $t\in\mathcal{V}$, that are closest to $\tau_t$ compared to any other time in $\mathcal{V}$.
If the are ties, we assign each item $i$ to exactly one, arbitrarily chosen, time in $\mathcal{V}$ that is closest to $d_i$.
For a given interval $t\in\mathcal{T}$, the algorithm schedules all items in $\mathcal{I}_t$ to be delivered at $\tau_t$.
This decision is appropriate for our purposes when $\sum_{i\in\mathcal{I}_t}\delta_i(\tau)\leq B$.
However, there is a solution quality degradation when $\sum_{i\in\mathcal{I}_t}\delta_i(\tau)>B$.
We manage this issue by modifying the bin delivery times in intervals of the latter type.
The refinement allows reducing the approximation ratio of the algorithm down to 2.

Partition $\mathcal{T}$ into the subsets $\mathcal{X}=\{t:m_t=1,t\in\mathcal{T}\}$ and $\mathcal{Y}=\{t:m_t\geq2,t\in\mathcal{T}\}$ of intervals containing exactly one and at least two bin deliveries in $\mathcal{S}$, respectively.
We select new bin delivery times for each interval $t\in\mathcal{Y}$.
Let $\mathcal{A}_t=\{i:d_i\leq\tau_t,i\in\mathcal{I}_t\}$ and $\mathcal{B}_t=\{i:d_i>\tau_t,i\in\mathcal{I}_t\}$ be the subsets of items with due times not later than and after $\tau_t$, respectively.
Also, set $n_t^{\mathcal{A}}=|\mathcal{A}_t|$ and $n_t^{\mathcal{B}}=|\mathcal{B}_t|$.
Denote by $\pi$ the sequence of items in $\mathcal{A}_t$, sorted in non-decreasing order of due times.
That is, $\mathcal{A}_t=\{\pi(1),\ldots,\pi(n_t^{\mathcal{A}})\}$ and $d_{\pi(1)}\leq\ldots\leq d_{\pi(n_t^{\mathcal{A}})}$.
Suppose that there exists an item $\pi(i)\in\mathcal{A}_t$ s.t.\ $\sum_{g=1}^{i}\delta_g(\tau_t)>B$.
Then, we pick the smallest indexed such $\pi(i)$ item in the sequence and add a bin delivery at $d_{\pi(i)}$ including exactly the items in $\{\pi(1),\ldots,\pi(i)\}$.
We repeat by considering the items $\{\pi(h),\ldots,\pi(n_t^{\mathcal{A}})\}$, where $h=i+1$.
That is, we identify a minimum index item $\pi(i')$ s.t.\ $\sum_{g=h}^{i'}\delta_{\pi(g)}(\tau_t)>B$, we add a bin delivery at $d_{\pi(i')}$ containing exactly the items $\{\pi(h),\ldots,\pi(i')\}$ and so on.
Next, we perform the same process for the $\mathcal{B}_t$ items by considering their non-increasing order $\xi$ of due times, i.e.\ $\mathcal{B}_t=\{\xi(1),\ldots,\xi(n_t^{\mathcal{B}})\}$ and $d_{\xi(1)}\geq\ldots\geq d_{\xi(n_t^{\mathcal{B}})}$.
Specifically, we begin by identifying a maximal index item $\xi(i)\in\mathcal{B}_t$ s.t.\ $\sum_{g=i}^{n_t^{\mathcal{B}}}\delta_{\xi(g)}(\tau_t)>B$, add a bin delivery at $d_{\xi(i)}$, etc.
For the remaining $\widetilde{\mathcal{A}}_t\subseteq\mathcal{A}_t$ and $\widetilde{\mathcal{B}}_t\subseteq\mathcal{B}_t$ items that have not been assigned to a bin delivery, it holds that $\sum_{i\in\widetilde{\mathcal{A}}_t}\delta_i(\tau_t)\leq B$ and $\sum_{i\in\widetilde{\mathcal{B}}_t}\delta_i(\tau_t)\leq B$.
If $\sum_{i\in\widetilde{\mathcal{A}}_t\cup\widetilde{\mathcal{B}}_t}\delta_i(\tau_t)\leq B$, then we assign all these items to a single bin delivery at $\tau_t$.
Otherwise, we use one bin delivery for the $\widetilde{\mathcal{A}}_t$ items and another one for the $\widetilde{\mathcal{B}}_t$ items.
Algorithm~\ref{Alg:Refinement} more formally describes this refinement of the delivery times during a $t\in\mathcal{Y}$ interval in a solution produced by the decouling algorithm.

\begin{algorithm}[h] \nonumber
\caption{$\text{Refinement}(t)$}
\begin{algorithmic}[1]
\STATE $\widetilde{\mathcal{A}}_t=\{i:d_i\leq\tau_t,i\in\mathcal{I}_t\}$
\STATE Sort the $\widetilde{\mathcal{A}}_t$ items s.t.\ $\widetilde{\mathcal{A}}_t=\{\pi(1),\ldots,\pi(n_t^{\mathcal{A}})\}$ and $d_{\pi(1)}\leq\ldots\leq d_{\pi(n_t^{\mathcal{A}})}$.
\STATE $h=1$
\WHILE {$\sum_{\pi(g)\in\widetilde{\mathcal{A}}_t}\delta_{\pi(g)}(\tau_t)>B$}
\STATE Find min index $\pi(i)\in\widetilde{\mathcal{A}}_t$ s.t.\ $\sum_{g=h}^i\delta_{\pi(g)}(\tau_t)>B$.
\STATE Schedule delivery at $d_{\pi(i)}$ with items $\pi(h),\ldots,\pi(i)$.
\STATE $\widetilde{\mathcal{A}}_t=\widetilde{\mathcal{A}}_t\setminus\{\pi(h),\ldots,\pi(i)\}$.
\STATE $h=i+1$.
\ENDWHILE
\STATE $\widetilde{\mathcal{B}}_t=\{i:d_i>\tau_t,i\in\mathcal{I}_t\}$
\STATE Sort the $\widetilde{\mathcal{B}}_t$ items s.t.\ $\widetilde{\mathcal{B}}_t=\{\xi(1),\ldots,\xi(n_t^{\mathcal{B}})\}$ and $d_{\xi(1)}\geq\ldots\geq d_{\xi(n_t^{\mathcal{B}})}$
\STATE $h=n_t^{\mathcal{B}}$
\WHILE {$\sum_{\xi(g)\in\widetilde{\mathcal{B}}_t}\delta_{\xi(g)}(\tau_t)>B$}
\STATE Find max index $\xi(i)\in\widetilde{\mathcal{B}}_t$ s.t.\ $\sum_{g=i}^h\delta_{\xi(g)}(\tau_t)>B$.
\STATE Schedule delivery at $d_{\xi(i)}$ with items $\xi(i),\ldots,\xi(h)$.
\STATE $\widetilde{\mathcal{B}}_t=\widetilde{\mathcal{B}}_t\setminus\{\xi(h),\ldots,\xi(i)\}$.
\STATE $h=i-1$.
\ENDWHILE
\IF {$\sum_{g\in\widetilde{\mathcal{A}}_t\cup\widetilde{\mathcal{B}}_t}\delta_g(\tau_t)>B$} 
\STATE Schedule bin delivery at $\tau_t$ for each among $\widetilde{\mathcal{A}}_t$ and $\widetilde{\mathcal{B}}_t$.
\ELSE
\STATE Schedule one bin delivery at $\tau_t$ for the $\widetilde{\mathcal{A}}_t\cup\widetilde{\mathcal{B}}_t$ items. 
\ENDIF
\end{algorithmic}
\label{Alg:Refinement}
\end{algorithm}

\begin{lemma}
\label{Lemma:Decoupling_Interval_Bound}
In a solution $\mathcal{S}$ produced by the refined decoupling algorithm, it holds that $m_t\leq \frac{1}{B}\sum_{i\in\mathcal{I}_t}\delta_i(\tau_t)+1$, for each time interval $t\in\mathcal{Y}$ with a delivery time at $\tau_t\in[a_t,b_t]$.
\end{lemma}
\begin{proof}
In a time interval $t\in\mathcal{Y}$, denote by $\mathcal{I}_t^+$ and $m_t^+$ the subset of items and the number of bins delivered at time $\tau_t$. 
Also, let $\mathcal{I}_t^-=\mathcal{I}_t\setminus\mathcal{I}_t^+$ and $m_t^-=m_t-m_t^+$ be the subset of items and number of bins delivered at any time during $[a_t,b_t]$ other than $\tau_t$.
For each bin delivery in $t$, we consider two cases based on whether it takes place at $\tau_t$, or not.

Initially, consider a bin delivery taking place during $[a_t,\tau_t)$. 
Clearly, $\sum_{i\in\mathcal{A}_t}\delta_i(\tau_t)>B$.
Suppose that the algorithm schedules its first bin delivery at time $d_{\pi(i)}$, where $i\in\{2,\ldots,n_t^{\mathcal{A}}-1\}$.
That is, right after the refinement, the items in $\mathcal{A}_t^-=\{\pi(1),\ldots,\pi(i)\}$ are delivered at $d_{\pi(i)}$, and the items in $\mathcal{A}_t^+=\{\pi(i+1),\ldots,\pi(n_t^{\mathcal{A}})\}$ remain to be scheduled.
Since the algorithm did not add a bin delivery at $d_{\pi(i-1)}$, we have $\sum_{g=1}^{i-1}\delta_{\pi(g)}(\tau_t)\leq B$.
Thus, the items $\{\pi(1),\ldots,\pi(i)\}$ can feasibly delivered all together at $d_{\pi(i)}$ using a single bin, since such a bin delivery would incur an inventory cost 
$\sum_{g=1}^{i-1}\delta_{\pi(g)}(d_{\pi(i)}) 
\leq \sum_{g=1}^{i-1}\delta_{\pi(g)}(\tau_t)\leq B$.
This argument holds for any refined bin delivery that does not take place at $\tau_t$.
So, each of the $m_t^-$ deliveries allows packing a subset $\mathcal{S}\subseteq\mathcal{I}_t$ of items that incur an inventory cost $\sum_{i\in\mathcal{S}}\delta_i(\tau_t)>B$ in the solution of the original decoupling algorithm, in a single bin delivery.
Thus, in the refined solution, we have that $m_t^-\leq\frac{1}{B}\sum_{i\in\mathcal{I}_t^-}\delta_i(\tau_t)$.

Now, consider the bin deliveries taking place at time $\tau_t$.
If $\widetilde{\mathcal{A}}_t$ and $\widetilde{\mathcal{B}}_t$ are the subsets of $\mathcal{A}_t$ and $\mathcal{B}_t$ items, respectively, delivered at $\tau_t$ in the refined solution, then it holds that $\sum_{i\in\mathcal{A}_t}\delta_i(\tau_t)\leq B$ and $\sum_{i\in\mathcal{B}_t}\delta_i(\tau_t)\leq B$, by construction.
If the items in $\widetilde{\mathcal{A}}_t\cup\widetilde{\mathcal{B}}_t$ fit in a single bin containing exactly those items, then
$m_t^+=1$. 
Otherwise, it must be the case that $\sum_{i\in\mathcal{I}_t}\delta_i(\tau_t)>B$, but using one bin delivery for the $\widetilde{\mathcal{A}}_t$ items and another for the $\widetilde{\mathcal{B}}_t$ items produces a feasible solution, i.e.\ $m_t^+=2$.
In both subcases, we have that $m_t^+\leq\frac{1}{B}\sum_{i\in\mathcal{I}_t^+}\delta_i(\tau_t)+1$.
We conclude that $m_t=m_t^-+m_t^+\leq\frac{1}{B}\sum_{i\in\mathcal{I}_t^-}\delta_i(\tau_t)+\frac{1}{B}\sum_{i\in\mathcal{I}_t^+}\delta_i(\tau_t)+1=\frac{1}{B}\sum_{i\in\mathcal{I}_t}\delta_i(\tau_t)+1$.
\end{proof}

\begin{theorem}
The refined decoupling algorithm is tightly $2$-approximate.
\end{theorem}
\begin{proof}
Consider an optimal solution $\hat{\mathcal{S}}$ with $\hat{m}$ bins and the schedule $\mathcal{S}$ with $m$ bins produced by the algorithm by calculating $k=\hat{m}$ bin delivery times.
Let $\mathcal{T}=\{1,2,\ldots,k\}$ be the time horizon partitioning obtained from the algorithm's solution and recall
the partition of $\mathcal{T}$ into the subsets $\mathcal{X}=\{t:m_t=1,t\in\mathcal{T}\}$ and $\mathcal{Y}=\{t:m_t\geq2,t\in\mathcal{T}\}$ containing exactly one and at least two bin deliveries in $\mathcal{S}$.
By definition, it must be the case that $\hat{m}=x+y$.
Given that $\hat{\mathcal{S}}$ is feasible and that the refined decoupling algorithm computes the initial delivery points so that the total inventory cost is minimal, we have that $\hat{m}\geq\frac{1}{B}\sum_{i\in\mathcal{I}}\delta_i(\mathcal{T})$.
On the other hand, we have that $m\leq x+\sum_{t\in\mathcal{Y}}m_t$.
By Lemma~\ref{Lemma:Decoupling_Interval_Bound}, $\sum_{t\in\mathcal{Y}}m_t\leq y+\frac{1}{B}\sum_{i\in\mathcal{I}}\delta_i(\mathcal{T})$.
Hence, $m\leq 2\hat{m}$.
The tightness of our analysis is derived by considering the instance described at the end of Section~\ref{Section:Median_Deliveries}.
\end{proof}

\section{Concluding Remarks}
\label{Section:Conclusion}


We introduce the container delivery scheduling problem as a model for jointly optimizing transportation and inventory costs when processing orders with collection in supply chains.
Our results demonstrate that delivering all orders on time may result in substantial transportation costs.
However, significantly better performance guarantees are achievable with more flexible delivery scheduling strategies that tolerate bounded storage and backlog costs.
Such bounds can be useful to manufacturers when negotiating costs of delay in customer orders.

Our main contribution is provably efficient algorithms for solving the problem.
We develop an 8/3-approximate algorithm based on a greedy sequential approach.
In addition, we propose an asymptotically tight 2-approximate algorithm by decoupling delivery scheduling decisions from assignments of items to bin deliveries.
A key advantage of the latter approach is the tighter bounds on the total inventory cost of the items.
The current manuscript has a focus on computing sequential solutions.
However, the non-refined version of the decoupling algorithm may compute nested solutions.
Exploring possible improvements of approximation bounds by exploiting nested and non-sequential solution structures is an intriguing direction for future work. 
We expect the bounds and insights presented in the current manuscript to be useful for this purpose.


The container delivery scheduling problem can be part of more complex optimization problems arising in supply chain distribution.
For example, the development of optimization and algorithmic approaches for effectively solving capacitated versions of the problem with additional volumetric dimensions would be of interest.
Finally, robust approaches would be particularly useful for problem instances arising in practice since problem solutions are highly sensitive to early or delayed deliveries, e.g.\ due to uncertainties occuring during shipping \cite{Letsios2021b}.


\bibliographystyle{elsarticle-harv} 
\bibliography{refs}


\appendix
\section{Nomenclature}

\begin{longtable}{l l l}
\caption{Core Notation}\\
\toprule
Name & Description \\
\midrule
\multicolumn{2}{l}{\bf Problem Definition} & \\
$\mathcal{I}$ & Set of items. \\
$n$ & Number of items. \\
$d_i$ & Due time of item $i$. \\
$v_i$ & Delivery time of item $i$. \\
$\delta_i(\tau)$ & Inventory cost $|d_i-\tau|$ of item $i$ if delivered at time $\tau$. \\
$\delta(\sigma,\tau)$ & Length $|\tau-\sigma|$ of interval between time points $\sigma$ and $\tau$ \\
$B$ & Container inventory bound. \\
$\mathcal{S}$ & Feasible solution, e.g.\ one computed by an algorithm. \\
$\hat{\mathcal{S}}$ & Optimal solution. \\
$m$ (or $\hat{m}$) & Number of container deliveries in $\mathcal{S}$ (resp.\ $\hat{\mathcal{S}}$). \\
$\mathcal{S}_j$ (or $\hat{\mathcal{S}}_j$) & Subset of items assigned to bin $j$ in $\mathcal{S}$ (resp.\ $\hat{\mathcal{S}}$). \\
$\mu_j$ (or $\hat{\mu}_j$) & Delivery time of bin $j$ in $\mathcal{S}$ (resp.\ $\hat{\mathcal{S}}$). \\
\midrule
\multicolumn{2}{l}{\bf Preliminaries} & \\
$\sigma, \tau$ & Auxiliary notation for delivery or other time instants. & \\
$\delta(\sigma,\tau)$ & Length of time interval between times $\sigma$ and $\tau$. & \\
$\delta_i(\tau)$ & Length of time interval between times $d_i$ and $\tau$. & \\
$a_j,b_j$ & Min and max due date among the items in $\mathcal{S}_j$. \\
$\mathcal{S}_j^-$ (or $\mathcal{S}_j^+$) & Subset of items in $\mathcal{S}_j$ with due date not later than $\mu_j$. \\
$n_j^-$ (or $n_j^+$) & Cardinality of set $\mathcal{S}_j^-$ (resp.\ $\mathcal{S}_j^+$). \\
$\langle\mathcal{A},\beta\rangle$ & 3-Partition instance. \\
\midrule
\multicolumn{2}{l}{\bf Sequential Algorithm Analysis (Negative Result, Section 3.1)} & \\
$\ell$ & Number of distinct due times. \\
$\tau_t$ & $t$-th distinct due time. \\
$\gamma_t$ & Time interval length between $\tau_{t-1}$ and $\tau_t$. \\
$n_t$ & Number of jobs with due date at $\tau_t$. \\
\midrule
\multicolumn{2}{l}{\bf Sequential Algorithm Analysis (Positive Results, Section 3.2)} & \\
$[a_t,b_t]$ or $t$ & Time interval between the earliest and latest due time \\
& among the items assigned to two consecutive bin \\
& deliveries $j=2t-1$ and $j+1=2t$, i.e.\ $j$ is odd. \\
$\mathcal{T}$ & Set of $[a_t,b_t]$ time intervals. \\
$m_t$ (or $\hat{m}_t$) & Number of bin deliveries in interval $[a_t,b_t]$ in $\mathcal{S}$ (or $\hat{\mathcal{S}}$). \\
$\lambda_t$ (or $\hat{\lambda}_t$) & Total inventory cost incurred by the $\mathcal{I}_t$ items in $\mathcal{S}$ (or $\hat{\mathcal{S}}$). \\
$\mathcal{X}$ & Subset of $\mathcal{T}$ intervals with $\hat{m}_t=0$. \\
$\mathcal{Y}$ & Subset of $\mathcal{T}$ intervals with $\hat{m}_t\geq1$. \\
$\mathcal{I}_t$ & Subset of items assigned to the bin deliveries during in $t$. \\
$\mathcal{A}_t$, $\mathcal{B}_t$ & Subsets of items assigned to each of the two bin \\
& deliveries during time interval $t$. \\
\midrule
\multicolumn{2}{l}{\bf Decoupling Algorithm Analysis (Sec.\ 4)} & \\
$\mathcal{V}$ & Set of dinstict delivery times computed in the $k$-th \\
& iteration of the algorithm. \\
$k$ & Cardinality of $\mathcal{V}$. \\
$\tau_t$ & $t$-th distinct delivery time. \\
$\delta_i(\mathcal{V})$ & Distance between $d_i$ and the closest time in $\mathcal{V}$. \\
$\Delta(\mathcal{V})$ & Total inventory cost if each item $i\in\mathcal{I}$ is delivered at the \\
& time $\tau$ which is the closest to $d_i$ among the times in $\mathcal{V}$. \\
$\mathcal{D}_t$ & Subset of items delivered at time $\tau_t$. \\
$\text{ET}(\mathcal{I},B,k)$ & Subproblem of computing $k$ times s.t.\ $\Delta(\mathcal{V})$ is minimized \\
& and no item has inventory cost more than $B$. \\
$\text{BP}(\mathcal{D}_t,B)$ & Classic bin packing problem with item sizes $\delta_i(\mathcal{V})$ \\
& and bin capacities $B$. \\
\bottomrule
\label{tbl:notation}
\end{longtable}

\end{document}